\documentclass[12pt, reqno]{amsart}
\usepackage{amsmath, amsthm, amscd, amsfonts, amssymb, graphicx, color}
\usepackage[bookmarksnumbered, colorlinks, plainpages]{hyperref}
\hypersetup{colorlinks=true,linkcolor=red, anchorcolor=green, citecolor=cyan, urlcolor=red, filecolor=magenta, pdftoolbar=true}

\newtheorem{theorem}{Theorem}[section]
\newtheorem{lemma}[theorem]{Lemma}
\newtheorem{proposition}[theorem]{Proposition}
\newtheorem{corollary}[theorem]{Corollary}
\theoremstyle{definition}

\theoremstyle{remark}
\newtheorem{remark}[theorem]{Remark}
\numberwithin{equation}{section}

\usepackage[active]{srcltx}

\usepackage{graphicx}

\usepackage{latexsym}
\usepackage{amsmath,amsthm}
\usepackage{amsfonts}
\usepackage{enumerate}
\usepackage{enumitem}

\usepackage{url}

\RequirePackage[OT1]{fontenc}

\usepackage{calrsfs}

\usepackage{hyperref}

\makeatletter
  \renewcommand\section{\@startsection {section}{1}{\z@}%
                                   {-\bigskipamount}%
                                   {\medskipamount}%
                                   {\large\bfseries
                                   \raggedright}}

  \renewcommand\subsection{\@startsection {subsection}{2}{\z@}%
                                   {-\medskipamount}%
                                   {\smallskipamount}%
                                   {\bfseries
                                   \raggedright}}
\makeatother

\newcommand{\dom}{\operatorname{dom}}
\newcommand{\codom}{\operatorname{codom}}
\newcommand{\gr}{\operatorname{gr}}

\renewcommand{\S}{\operatorname{\mathsf{S}\!}}
\renewcommand{\S}{\mathsf{S}}

\newcommand{\intr}[2]{\overline{#1,#2}}

\renewcommand{\le}{\leqslant}
\renewcommand{\ge}{\geqslant}

\renewcommand{\geq}{\geqslant}
\renewcommand{\succeq}{\succcurlyeq}

\renewcommand{\smallint}{\textstyle{\int}}

\newcommand{\vv}{\mathbf{v}}
\newcommand{\w}{\mathbf{w}}

\newcommand{\al}{\alpha}
\newcommand{\be}{\beta}
\newcommand{\ga}{\gamma}

\newcommand{\si}{\sigma}
\newcommand{\la}{\lambda}
\newcommand{\La}{\Lambda}

\newcommand{\vpi}{\varphi}
\renewcommand{\Psi}{\overline{\Phi}}
\newcommand{\Om}{\Omega}

\newcommand{\B}{\mathfrak{B}}
\newcommand{\BB}{\mathcal{B}}
\newcommand{\D}{\mathcal{D}}
\newcommand{\X}{\mathcal{X}}
\newcommand{\Y}{\mathcal{Y}}
\newcommand{\ZZ}{\mathcal{Z}}

\renewcommand{\L}{\mathcal{S}}
\newcommand{\M}{\mathrm{M}}
\newcommand{\NN}{\mathcal{N}}
\renewcommand{\NN}{\mathrm{N}}

\newcommand{\LD}{\mathcal{L}\!\mathcal{D}}
\renewcommand{\LD}{\mathcal{L}{\kern -1.9pt}\mathcal{D}}
\renewcommand{\LD}{\mathcal{D}}
\renewcommand{\LD}{\mathcal{L}{\kern -4pt}\mathcal{C}}
\renewcommand{\LD}{\mathcal{R}{\kern -3pt}\mathcal{C}}

\newcommand{\id}{\mathrm{id}}

\renewcommand{\geq}[1]{\overset{#1}{\succeq}}

\newcommand{\ii}[1]{\mathrm{I}\!\left\{#1\right\}}

\newcommand{\supp}{\operatorname{\mathrm{supp}}}

\renewcommand{\P}{\operatorname{\mathsf{P}}} 
\newcommand{\PP}{\operatorname{\mathsf{P}}}
\newcommand{\E}{\operatorname{\mathsf{E}}}
\newcommand{\Var}{\operatorname{\mathsf{Var}}}

\newcommand{\N}{\mathbb{N}}
\newcommand{\Z}{\mathbb{Z}}
\newcommand{\R}{{\mathbb{R}}}

\newcommand{\F}[1]{\mathcal{F}_+^{#1}}
\renewcommand{\F}{\mathcal{F}}
\renewcommand{\H}[1]{\mathcal{H}_+^{#1}}
\renewcommand{\H}{\mathcal{H}}

\newcommand{\G}[1]{\mathcal{G}^{#1}_+}
\renewcommand{\G}{\mathcal{G}}

\renewcommand{\PP}[1]{\mathcal{P}^{#1}_+}
\renewcommand{\PP}{\mathcal{P}}

\newcommand{\ta}{{\tilde{a}}}
\newcommand{\tb}{{\tilde{b}}}
\newcommand{\tf}{{\tilde{f}}}
\newcommand{\tg}{{\tilde{g}}}
\newcommand{\tI}{{\tilde{I}}}

\newcommand{\tw}{{\tilde{w}}}
\newcommand{\tww}{{\tilde{\w}}}
\newcommand{\tx}{{\tilde{x}}}
\newcommand{\tz}{{\tilde{z}}}

\newcommand{\hF}{{\hat\F}}
\newcommand{\hG}{{\hat\G}}

\newcommand{\near}[1]{\mathrel{\underset{#1}{\scalebox{1.1}{$\nearrow$}}}}
\newcommand{\sear}[1]{\mathrel{\underset{#1}{\scalebox{1.1}{$\searrow$}}}}

\renewcommand{\d}{\mathrm{d}}

\newcommand{\vp}{\varepsilon}

\begin{document}
\setcounter{page}{1}

\title[Cones of generalized monotone functions and dual cones]{Convex cones of generalized multiply monotone functions and the dual cones}

\author[I. Pinelis]{Iosif Pinelis
}

\address{$^{1}$ Department of Mathematical Sciences, 
Michigan Technological University, 
Houghton, Michigan 49931, USA.}
\email{\textcolor[rgb]{0.00,0.00,0.84}{ipinelis@mtu.edu}}





\begin{abstract}
Let $n$ and $k$ be nonnegative integers such that $1\le k\le n+1$. 
The convex cone $\F_+^{k:n}$ of all functions $f$ on an arbitrary interval $I\subseteq\R$ whose derivatives $f^{(j)}$ of orders $j=k-1,\dots,n$ are nondecreasing is characterized 
in terms of extreme rays of the cone $\F_+^{k:n}$. 
A simple description of the convex cone dual to $\F_+^{k:n}$ is given. 
These results are useful in, and were motivated by, applications in probability. 
In fact, the results are obtained in a more general setting with certain generalized derivatives of $f$ of the $j$th order in place of $f^{(j)}$. 
Somewhat similar results were previously obtained in the case when the left endpoint of the interval $I$ is finite, with certain additional integrability conditions; such conditions fail to hold in the mentioned applications. Development of substantially new methods was needed to overcome the difficulties. 
\end{abstract} 

\maketitle

\tableofcontents

\section{Introduction}\label{intro} 



In applications in probability (see e.g.\ \cite{
denuit-etal,denuit-etal99,normal,asymm,bent-ap,shaked-shanti,pin-hoeff-arxiv-reftoAIHP,
left-arxiv} and references therein), one is concerned with stochastic domination, defined by a formula of the form 
\begin{equation*}
	X\geq\F Y\overset{\operatorname{def}}\iff
	X\geq{}Y ({\operatorname{mod}}\ \F)
	\overset{\operatorname{def}}\iff\E f(X)\ge\E f(Y)\ \text{for all }f\in\F, 
\end{equation*}
where $X$ and $Y$ are random variables (r.v.'s) and $\F$ is a class of functions, assuming the expectations are appropriately defined. 

In the case when $\F$ is the class of all nondecreasing functions, the relation $\geq\F$ is called the first-order stochastic dominance. 

In general, the functions $f\in\F$ may be referred to as the test functions. 

Unless $\E f(X)=\E f(Y)=\infty$ or $\E f(X)=\E f(Y)=-\infty$, the inequality $\E f(X)\ge\E f(Y)$ can be rewritten as $\nu(f):=\int f\d\nu\ge0$, where $\nu$ is the signed measure (say $\nu_{X,Y}$) equal the difference between the probability distributions of the r.v.'s $X$ and $Y$. 

More generally, one may allow $\nu$ to belong to a larger set (say $\NN$) of signed measures on an interval $I$, which are not necessarily the differences between two probability measures. Usually, the set $\NN$ is assumed to be a convex cone. One can define the cone dual to $\F$ by the formula
\begin{equation*}
	\hat\F:=\{\nu\in\NN\colon\nu(f)\ge0\text{ for all }f\in\F\}.  
\end{equation*}
Thus, at least in the case when the r.v.'s $X$ and $Y$ are such that $\E f(X)$ and $\E f(Y)$ are both finite for all $f\in\F$, one will have $X\geq\F Y\iff\nu_{X,Y}\in\hat\F$. 
 
For any $\al$ and $\be$ in $\Z\cup\{\infty\}$, let $\intr\al\be:=\{j\in\Z\colon\al\le j\le\be\}$. 
In what follows, assume that the values of indices $i$, $j$, $k$, $\ell$, $m$, $n$ are each in the set $\intr0\infty$, unless specified otherwise. 

Classes of test functions of particular interest in the mentioned applications are 
\begin{equation}\label{eq:F unit}
	\F_+^{k:n}:=\F_+^{k:n}(I):=
	\big\{f\in\L^n\colon\text{$f^{(j)}$ is nondecreasing for each $j\in\intr{k-1}n$\,}\big\},  
\end{equation}
where 
\begin{equation}\label{eq:k le n+1}
	1\le k\le n+1,  
\end{equation}
$I$ is an interval in $\R$ with endpoints $a$ and $b$ such that $a<b$, 
$\L^n=\L^n(I)$ is the set of all $(n-1)$-times differentiable functions $f\in\R^I$ such that the function 
$f^{(n-1)}$ is absolutely continuous with an almost everywhere (a.e.) derivative (denoted here by $f^{(n)}$) that is (i) right-continuous on the interval $I\setminus\{b\}$ and (ii) left-continuous at the point $b$ in the case when $b\in I$.  
The above definition of the set $\L^n$ is actually valid only if $n\ge1$; let us complement it by letting $\L^0$ be the set of all Borel-measurable functions in $\R^I$. Then, in particular, $\F_+^{1:0}(I)$ will be the set of nondecreasing functions $f\in\R^I$. 

The functions of the class $\F_+^{n:n-1}$ are widely known as $n$-convex functions; see e.g.\ \cite{pec-prosch-tong
} and references therein. 

On the other hand, the nonnegative functions $g$ whose ``horizontal reflections'' $g(-\cdot)$ belong to the class $\F_+^{1:n-1}$ are also widely known as $n$-monotone functions;  see e.g.\ \cite{williamson56,levy62,
lefevre13_n-mono-s-conv}; in particular, in \cite{lefevre13_n-mono-s-conv} multiply monotone distributions are compared using $s$-convex stochastic orders -- on $[0,\infty)$, with four applications to insurance;  
other applications to insurance problems were given in \cite{denuit14}. Clearly, the notion of the $n$-monotonicity is an extension of Bernstein's complete monotonicity.  
 
Note that the class $\F_+^{k:n}$ \big(or, rather, its ``reflection'' $\F_-^{k:n}$ defined in \eqref{eq:F_-}\big) may be considered a generalization/extension of the class of all completely monotone functions on $(0,\infty)$ (in the Bernstein sense). Indeed, the latter class coincides with $\bigcap_{n=0}^\infty\F_-^{1:n}\big([0,\infty)\big)$; cf.\ Proposition~3.4 in \cite{pin-hoeff-arxiv-reftoAIHP} and its proof therein. 

The case of stochastic domination ${\operatorname{mod}}\ \F_+^{k:n}$ with $a>-\infty$ has been systematically studied in the literature; see e.g.\ \cite{karlin-novikoff,ziegler,karlin-studden,pec-prosch-tong,denuit-etal,denuit-etal99}. In this case, one can rely on the Taylor expansion 
\begin{equation}\label{eq:taylor-intro}
	f(x)=\sum_{i\in\intr0n}f^{(i)}(a+)\,\frac{(x-a)^i}{i!}
	+\int_{I
	}\d f^{(n)}(t)\frac{(x-t)_+^n}{n!}  
\end{equation}
for $x\in I$; as usual, we let $u_+:=0\vee u$ and $u_+^v:=(u_+)^v$ for all $u\in\R$ and $v\in[0,\infty)$, along with the convention $0^0:=1$. 
Note that $f^{(i)}\ge0$ and hence $f^{(i)}(a+)\ge0$ for any $f\in\F_+^{k:n}$ and any $i\in\intr kn$. 
It is also clear that the functions $I\ni x\mapsto c(x-a)^i$, $I\ni x\mapsto(x-a)^j$, and $I\ni x\mapsto(x-t)_+^n$ belong to the set $\F_+^{k:n}$ for all $c\in\R$, $i\in\intr0{k-1}$, $j\in\intr kn$, and $t\in I$. 
So, assuming appropriate integrability conditions, one has the following characterization of the dual cone $\hat\F_+^{k:n}$: a signed measure $\nu\in\NN$ is in $\hat\F_+^{k:n}$ if and only if all of the following three conditions hold: 
	\begin{enumerate}[label=(\roman*)]
	\item[(i)] 
$\int_I(x-a)^i\,\nu(\d x)=0$ 
for all $i\in\intr0{k-1}$;  
	\item [(ii)]
	\rule{0pt}{10pt}$\int_I(x-a)^j\,\nu(\d x)\ge0$ for all $j\in\intr kn$;   
	\item[(iii)] 
	\rule{0pt}{10pt}$\int_I(x-t)_+^n\,\nu(\d x)\ge0$ 
	for all $t\in I$. 
\end{enumerate}
Such a characterization of the dual cone is very useful, as it reduces 
the verification of the inequality $\nu(f)\ge0$ for all test functions $f\in\F_+^{k:n}$ to the verification of this inequality just in the case when $f$ is in a certain set of polynomials and their ``positive parts'' $x\mapsto(x-t)_+^n$. 
One may note that, in the case when $k\le n$ (cf.\ \eqref{eq:k le n+1}), the conjunction of the above conditions (i) and (ii) is equivalent to that of conditions 
	\begin{enumerate}
	\item[(i$'$)]\label{i'} 
$\int_I x^i\,\nu(\d x)=0$ 
for all $i\in\intr0{k-1}$;  
	\item [(ii$'$)]
	\rule{0pt}{10pt}$\int_I x^k\,\nu(\d x)\ge0$;   
	\item [(ii$''$)]
	\rule{0pt}{10pt}$\int_I(x-a)^j\,\nu(\d x)\ge0$ for all $j\in\intr{k+1}n$.   
\end{enumerate}

\bigskip
Alas, Taylor expansion \eqref{eq:taylor-intro} does not seem to make sense when $a=-\infty$ and $n\ge1$, and then the entire argument no longer holds; cf.\ e.g.\ \cite[Remark~3.6]{denuit-etal}. 


On the other hand, it is the case when $I=\R$ and hence $a=-\infty$ that is of foremost interest in the mentioned applications in probability \cite{
normal,bent-ap,pin-hoeff-arxiv-reftoAIHP,
left-arxiv}, as the distribution of the r.v.\ $X$ in those applications may be normal (as e.g.\ in \cite{normal}) or a convolution of a normal distribution and a Poisson one (as e.g.\ in \cite{pin-hoeff-arxiv-reftoAIHP}), whose support set will then be the entire real line, or with a support set bounded from above rather than from below, as e.g.\ in \cite{
normal,bent-ap,pin-hoeff-arxiv-reftoAIHP,
left-arxiv}.  
In general as well, it is desirable to allow the support sets of both $X$ and $Y$ not to be a priori bounded either from above or below. 


Because of the lack of the Taylor expansion \eqref{eq:taylor-intro}, it is much more difficult to obtain a characterization of the dual cone $\hat\F_+^{k:n}$ in the case when $a=-\infty$ and $k\ne n+1$. 
The first step here is to observe that for any $f\in\F_+^{k:n}$ one has $f^{(j)}(a+)=0$ for all $j\in\intr{k+1}n$, which results in
the following Taylor expansion of the function $f^{(k)}$ ``at the point $-\infty+:=(-\infty)+$'': 
\begin{equation}\label{eq:taylor-unit-f^k}
\begin{aligned}
	f^{(k)}(x_k)=&\sum_{i\in\intr kn}f^{(i)}(-\infty+)\,\frac{(x_k+\infty)^{i-k}}{(i-k)!}
	+\int_I\d f^{(n)}(t)\frac{(x_k-t)_+^{n-k}}{(n-k)!} \\ 
	=&f^{(k)}(-\infty+)
	+\int_I\d f^{(n)}(t)\frac{(x_k-t)_+^{n-k}}{(n-k)!}  
\end{aligned}	
\end{equation}
for $x_k\in I$; here and elsewhere, we are assuming the conventions $0\cdot c:=0$ for any $c\in[-\infty,\infty]$ and $\infty^0:=1$. 
Next, we fix an arbitrary $z\in Y$ and, for any $y\in I\cap(-\infty,z)$, truncate the above Taylor expansion of the $k$th derivative $f^{(k)}$ by replacing the integral $\int_I$ in \eqref{eq:taylor-unit-f^k} with $\int_{I\cap[y,\infty)}$; let us denote the resulting function by 
$(f^{(k)})_y$. 
Finally, the so truncated $k$th derivative is lifted back up, in the sense that a function $g_y$ is constructed so that the conditions $(g_y)^{(k)}=(f^{(k)})_y$ and $(g_y)^{(i)}(z)=f^{(i)}(z)$ for all $i\in\intr0{k-1}$ hold. In fact, $g_y$ is completely determined by these conditions and is given by  formulas \eqref{eq:g_y=}, \eqref{eq:P_zy}, and \eqref{eq:R_zy}. 
Moreover, $g_y$ approximates $f$ in the sense of \eqref{eq:betw}. 
So, $g_y$ may be considered an approximate Taylor expansion of $f$ at $a=-\infty$. 
As Remark~\ref{rem:left} shows, in general functions $f\in\F_+^{k:n}$ admit only of such an approximate Taylor expansion of $f$ at $a=-\infty$; that is, one cannot do without the truncation described above. 

However, this approximate Taylor expansion of $f$ is enough to obtain a desired characterization of the dual cone $\hat\F_+^{k:n}$ in the case when $a=-\infty$ and $k\ne n+1$, 
which is as follows: 
a signed measure $\nu\in\NN$ is in $\hat\F_+^{k:n}$ if and only if 
	\begin{enumerate}[label=(\roman*$_{-\infty}$),itemsep=1ex,leftmargin=1.5cm]
	\item
	\label{i_-infty}
$\int_I x^i\,\nu(\d x)=0$ 
for all $i\in\intr0{k-1}$;  
	\item 
	$\int_I x^k\,\nu(\d x)\ge0$;   
	\item
	$\int_I(x-t)_+^n\,\nu(\d x)\ge0$ 
	for all $t\in I$. 
\end{enumerate} 

%
%


One can see that conditions (i$_{-\infty}$), (ii$_{-\infty}$), and (iii$_{-\infty}$) 
are, respectively, the same as conditions (i$'$), (ii$'$), and (iii) on page~\pageref{i'}; however, condition (ii$''$) from page~\pageref{i'} ``disappears'' when $a=-\infty$. 

The case when $k=n+1$ is overall simpler (than the just discussed case $k\ne n+1$) but has a certain peculiarity to it, to be addressed later in this paper. 


In fact, we consider a more general version of the class $\F_+^{k:n}$, defined in \eqref{eq:F unit}, by replacing the operators $f\mapsto f^{(j)}$ of multiple differentiation with more general differential operators, including ones of the form 
\begin{equation}\label{eq:E^j,intro}
	E^j:=E^j_{w_0,\dots,w_j}:=(R_{w_j}D)\cdots(R_{w_1}D)R_{w_0}, 
\end{equation}
where $D$ is the usual differentiation operator, $w_0,w_1,\dots$ are positive smooth enough functions, and $R_wf:=f/w$ for any function $f\in\R^I$. 
Thus, the operator $E^j$ is the alternating composition of the operators of the division by positive functions and the differentiation operator. 
The functions $w_0,w_1,\dots$ may be referred to as the \emph{gauge} functions. 
In the \emph{unit-gauges} case, with $w_0=w_1=\dots=1$, the operator $E^j_{w_0,\dots,w_j}$ reduces back to $D^j$, the operator of the $j$-fold differentiation.

As will be shown in Proposition~\ref{prop:invar}, a special case of non-unit gauge functions -- with $w_0=1$ and $w_1=w_2=\dots=\psi'$ for a general continuous function $\psi'$ -- arises from the unit gauges by (generally nonlinear) change of scale.  

A reason to consider general gauge functions $w_0,w_1,\dots$ is to encompass, in particular, the corresponding results in \cite{karlin-novikoff,ziegler,karlin-studden} on dual cones, defined in terms of extended complete Tchebycheff systems. 
Details on this are given in Subsection~\ref{tcheb}. 
The theory and applications of Tchebycheff systems have a long and rich history; see e.g.\ \cite{karlin-studden,krein-nudelman
}. 

Dealing with general, not necessarily unit gauge functions $w_0,w_1,\dots$ requires overcoming more difficulties. One of them is that such an explicit representation of the approximation $g_y$ of $f$ as the one mentioned above and given by formulas \eqref{eq:g_y=}, \eqref{eq:P_zy}, and \eqref{eq:R_zy} for the unit-gauges case is then no longer available. 
Here, to be used in place of usual polynomials, generalized polynomials are introduced, depending on the sequence $\w:=(w_0,w_1,\dots)$ of gauge functions; rather naturally, a function $p$ is called a 
$\w$-polynomial of degree $j$ if the function $E^j_{w_0,\dots,w_j}p$ is a nonzero constant.  

Another notable distinction from the unit-gauges case is that in general, in place of the set $\intr{k+1}n$ in condition (ii$''$) on page~\pageref{i'} 
for the unit-gauges case, one may get any given subset of $\intr{k+1}n$, depending on the choice of the gauge functions $w_0,w_1,\dots$, as follows from Proposition~\ref{prop:M}. 
No phenomenon of this kind appears to have been observed before. 

%

Also, in distinction with \cite{karlin-novikoff,ziegler,karlin-studden}, we impose no smoothness conditions on the gauge functions $w_j$, except for being Borel-measurable. Accordingly, the entries of the differentiation operator $D$ in the definition \eqref{eq:E^j,intro} need to be slightly modified, with some extra care exercised in the definition of the composition of operators; cf.\ \eqref{eq:D^j,E^j} and the paragraph containing formula \eqref{eq:TS}. 

Closely related moment problems for generalized polynomials on a semi-infinite interval in $\R$ and on $\R$ itself were considered by Kre{\u\i}n and Nudel$'$man \cite[Chapter~V]{krein-nudelman}. Essentially, the method used there is compactification of the (semi-)infinite interval -- which, however, requires additional restrictions on the limit behavior of certain generalized polynomials or their ratios near the infinite endpoint(s). 
No such additional restrictions are assumed in the present paper. 


The paper is organized as follows. In Sections~\ref{gauged}--\ref{cones} we develop necessary, mostly quite novel tools in order to provide a convenient enough description of the cone $\F_+^{k:n}$ of generalized monotone functions. These developments, listed in the table of contents, culminate in 
formulas \eqref{eq:betw} and \eqref{eq:g_y in}, according to which every function $f\in\F_+^{k:n}$ is approximated in a monotone manner by a function $g_y$, which is the sum of two summands: (i) a member of a certain set of generalized polynomials and (ii) a mixture of the ``positive parts'' (defined by formula \eqref{eq:p^+}) of generalized polynomials, belonging to another set. This new approximative representation of the functions $f\in\F_+^{k:n}$ enables us to provide a description, in Section~\ref{dual}, of the cones dual to the cones $\F_+^{k:n}$, for any subinterval $I$ of $\R$ and any gauge functions. This description of the dual cones is quite convenient in the desired applications and looks quite similar to the known descriptions of this kind, such as the mentioned ones in \cite{karlin-novikoff,ziegler,karlin-studden}, which latter were obtained assuming $a>-\infty$ and/or certain integrability conditions. However, without such additional conditions, quite substantial difficulties needed to be overcome.  
The  close relations of our results with the Tchebycheff systems are discussed in 
Subsection~\ref{tcheb}. Applications are briefly discussed in Section~\ref{appls}.

\section{Compositions of operators of gauged differentiation}\label{gauged}
Compositions of operators of gauged differentiation were mentioned above; recall formula \eqref{eq:E^j,intro}. 
In order to accurately define such compositions, 
let us first carefully define the composition of arbitrary maps, without any restrictions on their domains or codomains. Therefore, we shall need to be quite clear about 
the notion of a map.  

As usual, we say that $T$ is a map (equivalently, a mapping or a function or an operator) if $T$ is a triple of the form $(\X,\Y,\G)$, where $\X$ and $\Y$ are any sets and $\G$ is any subset of the set $\X\times\Y$ such that for each $x\in\X$ there is a unique $y\in\Y$ such that $(x,y)\in\G$. Thus, the sets $\X=:\dom T$, $\Y=:\codom T$, and $\G=:\gr T$ are attributes of the map $T$, called the domain, codomain, and graph of $T$, respectively. 
The identity map $\id_\X$ is the triple $(\X,\X,\D_\X)$ with $\D_\X:=\{(x,x)\colon x\in\X\}$. 

We write $T\colon\X\to\Y$ to mean that $T$ is a map with $\dom T=\X$ and $\codom T=\Y$;  
let, as usual, $\Y^\X$ denote the set of all such maps $T$.  

For any map $T$ and any $x\in\dom T$, $T(x)$ denotes the unique $y\in\codom T$ such that $(x,y)\in\gr T$. As usual, if $T$ is a linear map, let us write $Tx$ instead of $T(x)$. 
Note that it is not required for the set $\codom T$ to be the same as the image $T(\dom T):=\{T(x)\colon x\in\dom T\}$ of $\dom T$ under $T$. 

Given two maps, say $T_1$ and $T_2$, let us write $T_1\subseteq T_2$ if $\gr T_1\subseteq\gr T_2$ and $\codom T_1\subseteq\codom T_2$. In particular, $T_1\subseteq T_2$ implies   
$\dom T_1\subseteq\dom T_2$. 
Clearly, $T_1=T_2$ if and only if $T_1\subseteq T_2$ and $T_2\subseteq T_1$.


Let now $S\colon\X\to\Y_1$ and $T\colon\Y_2\to\ZZ$ be two maps; the sets $\Y_1$ and $\Y_2$ may differ from each other. The composition of $S$ and $T$ is denoted as $T\circ S$ \big(or simply as $TS$ if there is hardly a chance to confuse the composition with a pointwise product of functions\big) and defined by the conditions 
\begin{equation}
\text{$T\circ S\colon S^{-1}(\Y_2)\to\ZZ$\quad and\quad $(T\circ S)(x)=T(S(x))$ for all $x\in S^{-1}(\Y_2)$,}	
\end{equation}
where $S^{-1}(\Y_2):=\{x\in\X\colon S(x)\in\Y_2\}$. Thus, for any $x$ and $z$ one has 
\begin{equation}\label{eq:TS}
	(T\circ S)(x)=z\iff\big(S(x)\in\Y_2\,\ \&\ \,T(S(x))=z\big).
\end{equation} 
By the above definition, the composition $T\circ S$ 
and, in particular, its domain and codomain 
are completely determined by $S$ and $T$; in fact, the domain $S^{-1}(\Y_2)$ of the composition $T\circ S$ is already completely determined by $S$ and the domain $\Y_2$ of $T$, whereas the codomain $\ZZ$ of $T\circ S$ is the same as that of $T$. 

Now one can define, by induction, the composition of any finite number of any maps: 
$T_n\circ\cdots\circ T_1:=T_n\circ\cdots\circ T_3\circ(T_2\circ T_1)$ for any $n\in\intr3\infty$. 
Then, in particular, the domain and codomain of $T_n\circ\cdots\circ T_1$ are completely determined by maps $T_1,\dots, T_n$.  

If $S\colon\X\to\Y$, then the inverse map $S^{-1}$ is any map $T$ such that $T\colon\Y\to\X$, $T\circ S=\id_\X$, and $S\circ T=\id_\Y$. Clearly, any map has no more than one inverse. 

Take now any interval $I\subseteq\R$ of nonzero length; a particular possibility is that $I=\R$. 
Let 
\begin{equation*}
	a:=\inf I\quad\text{and}\quad b:=\sup I, 
\end{equation*}
%
so that 
\begin{equation}\label{eq:a,b}
-\infty\le a<b\le\infty. 
\end{equation}



Let $\BB$ denote the set of all Borel-measurable 
functions in $\R^I$. 
Then let 
\begin{equation*}
	\BB^+:=\{w\in\BB\colon w>0\}  
\end{equation*} 
and, for each $w\in\BB^+$,  define the linear operator $R_w$ by the conditions 
\begin{equation}\label{eq:R_w}
	R_w\colon\BB\to\BB\quad\text{and\quad}R_w f:=\frac fw\text{ for all }f\in\BB.  
\end{equation}

Let the ligature $\LD=\LD(I)$ denote the set of all functions in $\R^I$ that are (i) right-continuous on the interval $I\setminus\{b\}$ and (ii) left-continuous at the point $b$ in the case when $b\in I$. 
It is easy to see that $\LD\subseteq\BB$. 

Take now any $w\in\BB^+$. Let 
\begin{equation}\label{eq:LD_w}
\LD_w:=\big\{f\in\BB\colon R_w f\in\LD\big\}=\Big\{f\in\BB\colon \frac fw\in\LD\Big\}.   	
\end{equation} 

Next, introduce 
\begin{equation}\label{eq:D_w}
\begin{aligned}
	\L_w:=&\text{ the set of all functions $f\in\BB$ 
such that} \\ 
&\text{ there is a function $D_w f\in\LD_w$ satisfying the condition} \\ 
	&\ f(x)=f(z)+\int_z^x(D_w f)(u)\d u
	\quad\text{for all $x$ and $z$ in $I$. }
\end{aligned}	
\end{equation}
%
Here and elsewhere, $\int_z^x:=-\int_x^z$ if $x<z$. 
For any $f\in\L_w$, the function $f$ is continuous and even absolutely continuous, and a derivative of $f$ exists almost everywhere (a.e.) and coincides with $D_w f$ a.e. 
Therefore, in view of the condition $D_w f\in\LD_w$, 
the ``generalized derivative'' function $D_w f$ is uniquely determined, for each $w\in\BB^+$ and each $f\in\L_w$. 
Thus, one has the linear operator 
\begin{equation*}
	D_w\colon\L_w\to\BB. 
\end{equation*}
Here it may be noted that, for any two functions $w$ and $v$ in $\BB^+$ with $\{\frac vw,\frac wv\}\subseteq\LD$, 
one has $\LD_w=\LD_v$, whence $D_w=D_v$ (and, in particular, $\L_w=\L_v$). 

Let now $\w:=(w_0,w_1,\dots)$ be a sequence of 
locally bounded 
functions 
in $\BB^+$; 
the term ``locally bounded (on $I$\,)'' means ``bounded on any compact subset of $I$\,''. 


For each $j\in\intr0\infty$, let 
\begin{alignat}{2}
	& D_j:=D_{\w,j}:=D_{w_{j+1}}R_{w_j},\quad && \text{so that }D_j\colon R_{w_j}^{-1}(\L_{w_{j+1}})\to\BB, \label{eq:D_j} \\ 
	\intertext{and for each $j\in\intr1\infty$, let }
	& E_j:=E_{\w,j}:=R_{w_j}D_{w_j},\quad && \text{so that }E_j\colon\L_{w_j}\to\BB. \label{eq:E_j} 
\end{alignat}



Further, let 
\begin{equation}\label{eq:^0}
	\L^0:=\L_\w^0:=\BB,\quad D^0:=D_\w^0:=\id_\BB,\quad\text{and}\quad E^0:=E_\w^0:=R_{w_0}, 
\end{equation}
where $\id_\BB\colon\BB\to\BB$ is the identity operator, so that $D^0f=f$ for all $f\in\BB$. 
Thus, one has the linear operators 
\begin{equation*}
	\text{$D^0\colon\L^0\to\BB$\quad and\quad $E^0\colon\L^0\to\BB$.}   
\end{equation*}
Now for all $j\in\intr1\infty$ define the linear operators $D^j$\ and\ $E^j$ 
recursively by the formulas 
%
\begin{equation}\label{eq:D^j,E^j}
\begin{alignedat}{2}	
D^j:=D_\w^j :=&D_{j-1}D^{j-1}=D_{j-1}\cdots && D_0 
=D_{w_j}R_{w_{j-1}}
\cdots D_{w_1}R_{w_0},   \\  
E^j:=E_\w^j:=&E_jE^{j-1}=E_j\cdots E_1E^0   
&&=R_{w_j}D_{w_j}
\cdots R_{w_1}D_{w_1}R_{w_0} \\   
& &&=R_{w_j}D^j;       
\end{alignedat}
\end{equation}
it follows that $D^j$\ and\ $E^j$ have the same domain,  
\begin{equation}\label{eq:S^j}
\begin{aligned}
\L^j:=\L_\w^j:=&\{f\in\L^{j-1}\colon D^{j-1}f\in R_{w_{j-1}}^{-1}(\L_{w_j})\} \\ 
=&\{f\in\L^{j-1}\colon E^{j-1}f\in\L_{w_j}\},   
\end{aligned}	
\end{equation}
and the same codomain, $\BB$. 
%
%
%
Thus, for each $j\in\intr0\infty$ one has the linear operators 
\begin{equation}\label{eq:domain}
	\text{$D^j\colon\L^j\to\BB$\quad and\quad $E^j\colon\L^j\to\BB$.}   
\end{equation}
It also follows from \eqref{eq:S^j} that 
\begin{equation}\label{eq:supseteq}
	\L^0
	\supseteq\L^1\supseteq
	\cdots. 
\end{equation}

Introduce now the notation 
\begin{equation}\label{eq:f^j}
	f^{(j)}:=f_\w^{(j)}:=E^j f=\frac{D^jf}{w_j} 
\end{equation}
for $j\in\intr0\infty$ and 
$f\in\L^j$. 
Note that   
\begin{equation}\label{eq:f^(j)more}
\begin{aligned}
	&f\in\L^n\implies \\ 
	&\quad\quad\quad\quad\quad
	\left\{
	\begin{aligned}
	&f^{(0)},\dots,f^{(n-1)}\text{ are absolutely continuous, }f^{(n)}\in\LD, \\ 
	&f=f^{(0)}w_0, 
	\\ 
	&f^{(j)}(x)=f^{(j)}(z)+\int_z^x f^{(j+1)}(u)w_{j+1}(u)\d u \\ 
	&\qquad\qquad\qquad\text{ for all }j\in\intr0{n-1}, x\in I, z\in I. 
	\end{aligned}
	\right.
\end{aligned}	
\end{equation}
  
The functions $w_j$ may be referred to as the \emph{gauge} functions. 

Concerning these functions, 
the simplest and most common case is when $w_j=1$ for all $j
$, which may be referred to as 
the \emph{unit-gauges} case. 
In that case, for each $n\in\intr1\infty$ the set $\L^n$ coincides with the set of all $(n-1)$-times  differentiable functions $f\colon I\to\R$ such that the function 
$f^{(n-1)}$ is absolutely continuous with an a.e. derivative coinciding a.e.\ with a function in $\LD$, and at that for each $j\in\intr0{n-1}$ each of the ``gauged'' higher-order derivatives $f^{(j)}$ and $D^jf$ coincides with the usual $j$th derivative of $f$. 



The generalized derivatives $f^{(j)}$ have the following simple but important invariance property with respect to (generally nonlinear) change of scale. 


\begin{proposition}\label{prop:invar} 
Let $\psi\colon\tI\to I$ be a continuously differentiable function with $\psi'>0$, mapping some interval $\tI$ with endpoints $\ta$ and $\tb$ onto the interval $I$. \big(If $\ta\in\tI$, then $\psi'(\ta)$ is of course understood as the right derivative of $\psi$ at $\ta$, and similarly for $\tb$.\big) 
Let $\tww:=(\tw_0,\tw_1,\dots)$, where 
\begin{equation}\label{eq:tw}
\text{$\tw_0:=w_0\circ\psi$\quad and\quad $\tw_j:=(w_j\circ\psi)\,\psi'$ for $j\in\intr1\infty$.} 	
\end{equation}
Then for each $j\in\intr0\infty$ the following bipartite statement is true: 
\begin{enumerate}[label=\emph{(\Roman*)}]
	\item for each $f\in\L_\w^j$, one has 
\begin{equation*}
	f\circ\psi\in\L_\tww^j\quad\text{and}\quad
	(f\circ\psi)^{(j)}_\tww=f^{(j)}_\w\circ\psi;  
\end{equation*}
	\item	the map $\L_\w^j\ni f\longmapsto 
	f\circ\psi\in\L_\tww^j$ is bijective. 
\end{enumerate} 
\end{proposition}

The proof of Proposition~\ref{prop:invar} is based on 

\begin{lemma}\label{lem:invar}
Let $\tI$ and $\psi$ be as in Proposition~\ref{prop:invar}. 
Let $\tilde\BB$ denote the set of all Borel-measurable 
functions in $\R^\tI$, and then let $\tilde\BB^+:=\{\tw\in\tilde\BB\colon\tw>0\}$.    
Define the operators 
\begin{equation}\label{eq:C,K:}
C_\psi\colon\BB\to\tilde\BB\quad{and}\quad K_\psi\colon\BB^+\to\tilde\BB^+ 	
\end{equation}
using the formulas 
\begin{equation}\label{eq:C,K}
	C_\psi f:=f\circ\psi\quad\text{and}\quad K_\psi(w):=(w\circ\psi)\,\psi'. 
\end{equation}
Fix now any $w\in\BB^+$ and let 
\begin{equation}\label{eq:tw:=}
	\tw:=K_\psi(w). 
\end{equation}
Let also 
\begin{alignat}{2}
&	E_w:=R_wD_w,\quad && \text{so that }E_w\colon\L_w\to\BB, \label{eq:E_w} \\ 
&	E_\tw:=R_\tw D_\tw,\quad && \text{so that }E_\tw\colon\L_\tw\to\tilde\BB;  \notag 
\end{alignat}
cf.\ \eqref{eq:E_j}. 
Then one has the following commutation relations: 
\begin{equation}\label{eq:0-comm}
	C_\psi R_w=R_{C_\psi w} C_\psi   
\end{equation}
and
\begin{equation}\label{eq:1-comm}
	C_\psi E_w=E_\tw C_\psi.  
\end{equation}
\end{lemma}

\begin{proof}[Proof of Lemma~\ref{lem:invar}]
Identity \eqref{eq:0-comm} follows immediately from the definitions of $C_\psi$ in \eqref{eq:C,K} and $R_w$ in \eqref{eq:R_w}. 

Take now any $f\in\L_w$. 
For brevity, let $\tf:=C_\psi f=f\circ\psi$. 
Take also any $\tx$ and $\tz$ in $\tI$, and let $x:=\psi(\tx)$ and $z:=\psi(\tz)$, so that $x$ and $z$ are in $I$. 
Then, using \eqref{eq:D_w}, the change of the integration variable by the formula $u=\psi(v)$, \eqref{eq:E_w}, \eqref{eq:tw:=}, and \eqref{eq:C,K}, one has 
\begin{equation}\label{eq:5 lines}
\begin{aligned}
	\tf(\tx)-\tf(\tz)
	&=f(x)-f(z) \\ 
	&=\int_z^x(D_w f)(u)\,\d u \\  
	&=\int_\tz^\tx (D_w f)(\psi(v))\, \psi'(v) \,\d v \\ 
	&=\int_\tz^\tx (E_w f)(\psi(v))\,w(\psi(v))\, \psi'(v) \,\d v \\ 
	&=\int_\tz^\tx ((E_w f)\circ\psi)(v)\,\tw(v)\, \d v.  
\end{aligned}
\end{equation}	
Also by \eqref{eq:D_w}, $D_w f\in\LD_w$, so that, by \eqref{eq:E_w} and \eqref{eq:LD_w}, $E_w f\in\LD=\LD(I)$, which yields $(E_w f)\circ\psi\in\LD(\tI)$. 
So, by \eqref{eq:R_w}, $((E_w f)\circ\psi)\,\tw\in\LD_\tw$ and hence, again by \eqref{eq:D_w}, $\tf\in\L_\tw$ and $D_\tw\tf=((E_w f)\circ\psi)\,\tw$. 
Therefore, using again \eqref{eq:E_w}, one concludes that $E_\tw\tf=(E_w f)\circ\psi$. 
Thus, we have established the implication 
\begin{equation*}
	f\in\L_w\implies f\circ\psi=\tf\in\L_\tw\quad \&\quad E_\tw(f\circ\psi)=(E_w f)\circ\psi. 
\end{equation*}
In view of \eqref{eq:E_w}, this implications means that 
\begin{equation}\label{eq:1-comm,half}
	C_\psi E_w\subseteq E_\tw C_\psi=E_{K_\psi(w)} C_\psi;     
\end{equation}
cf.\ \eqref{eq:1-comm}. 
Next, note that 
\begin{equation*}
K_{\psi^{-1}}(\tw)=(\tw\circ\psi^{-1})(\psi^{-1})'
=(w\circ\psi\circ\psi^{-1})\,(\psi'\circ\psi^{-1})\,(\psi^{-1})'
=w.   	
\end{equation*}
Also, it is easy to check that $(C_\psi)^{-1}=C_{\psi^{-1}}$. 
So, the other ``half'' of \eqref{eq:1-comm}, $E_\tw C_\psi\subseteq C_\psi E_w$, can be rewritten as $C_{\psi^{-1}}E_\tw\subseteq E_{K_{\psi^{-1}}(\tw)} C_{\psi^{-1}}$, which follows by \eqref{eq:1-comm,half} if one replaces there $w$ and $\psi$ by $\tw$ and $\psi^{-1}$, respectively. 
This completes the proof of identity \eqref{eq:1-comm}, as well as that of entire Lemma~\ref{lem:invar}. 
\end{proof}

\begin{proof}[Proof of Proposition~\ref{prop:invar}] Take any $j\in\intr0\infty$. Then, 
by \eqref{eq:D^j,E^j}, \eqref{eq:E_w}, \eqref{eq:1-comm}, \eqref{eq:0-comm}, \eqref{eq:tw:=}, \eqref{eq:C,K}, and \eqref{eq:tw}, 
\begin{equation}\label{eq:C E^j = E^j C}
	C_\psi E_\w^j=C_\psi E_{w_j}\cdots E_{w_1}R_{w_0}=
	E_{\tw_j}\cdots E_{\tw_1}R_{\tw_0}C_\psi=E_\tww^j C_\psi.   
\end{equation}
The main idea here is to use successively the commutation relations \eqref{eq:1-comm} ($j$ times) and \eqref{eq:0-comm} (once) to obtain the second equality in \eqref{eq:C E^j = E^j C}. 

Take now any $f\in\L_\w^j$, as in part~(I) of Proposition~\ref{prop:invar}.  
By \eqref{eq:domain} and \eqref{eq:C,K:}, $C_\psi E_\w^j\colon\L_\w^j\to\tilde\BB$. 
So, by \eqref{eq:C E^j = E^j C}, 
\begin{equation}\label{eq:E^j C:}
	E_\tww^j C_\psi\colon\L_\w^j\to\tilde\BB.  
\end{equation}
Therefore, using again \eqref{eq:domain}, one has $f\circ\psi=C_\psi f\in\L_\tww^j$. 
Moreover, 
$(f\circ\psi)^{(j)}_\tww=E_\tww^j C_\psi f=C_\psi E_\w^j f=f^{(j)}_\w\circ\psi$; the second equality here follows again by \eqref{eq:C E^j = E^j C}. 
This verifies part~(I) of Proposition~\ref{prop:invar}. 

The injectivity of the map in part~(II) of Proposition~\ref{prop:invar} is due to the fact that $\psi$ is surjective. 

Finally, take any $g\in\L_\tww^j$. Let $f:=C_{\psi^{-1}}g=g\circ\psi^{-1}$. Then $g=C_\psi f$ and hence $E_\tww^j g =E_\tww^j C_\psi f$. So, by \eqref{eq:E^j C:}, $f\in\L_\w^j$. At that, $f\circ\psi=C_\psi f=g$.  
This verifies the surjectivity of the map in part~(II) of Proposition~\ref{prop:invar}, which completes the proof of the proposition. 
\end{proof}

\begin{remark}\label{rem:invar}
Let us say that two intervals $I$ and $\tI$ are equivalent if they are related via such a map $\psi$ as in Proposition~\ref{prop:invar}. Then there are only four equivalence classes, determined by which of the following four conditions holds: 
(i) $I\cap\{a,b\}=\emptyset$, (ii) $I\cap\{a,b\}=\{a\}$, (iii) $I\cap\{a,b\}=\{b\}$, or (iv) $I\cap\{a,b\}=\{a,b\}$; here $a$ and $b$ are as in \eqref{eq:a,b}. 
%

Therefore, since our main results will all be in terms of the generalized derivatives $f^{(j)}$, it would in principle be enough to assume that $I$ is one of the following four intervals: $\R$, $(-\infty,0]$, $[0,\infty)$, or $[0,1]$. 
We shall use this idea in the 
proof of Proposition~\ref{prop:M} and in Remark~\ref{rem:left}. 
However, most of our considerations will be applicable to all of these four kinds intervals, and so, it would be comparatively inefficient to deal with each kind of intervals separately. 
\end{remark}

\begin{center}
	***
\end{center}

Let $\S^i$ denote the $i$th power of the left-shift operator (say $\S$), so that 
\begin{equation}\label{eq:S}
	\S=\S^1\quad\text{and}\quad
	\S^i\w=\vv,\ \text{where}\ \vv=(v_0,v_1,\dots)=(w_i,w_{i+1},\dots). 
\end{equation}

The following proposition 
allows one to compare the values of two functions on an interval given a comparison between their gauged higher-order derivatives and the same ``initial'' conditions at a point of the interval. 


\begin{proposition}\label{prop:compar}
  Take any $z\in I$ and any $k\in\intr0\infty$. Suppose that functions $f$ and $g$ in $\L^k$ are such that $f^{(j)}(z)=g^{(j)}(z)$ for all $j\in\intr0{k-1}$. 
  Then the inequality $f^{(k)}\ge g^{(k)}$ on $I\cap[z,\infty)$ implies that $f\ge g$ on $I\cap[z,\infty)$. 
  Similarly, the inequality $f^{(k)}\ge g^{(k)}$ on $I\cap(-\infty,z]$ implies that $(-1)^k(f-g)\ge0$ on $I\cap(-\infty,z]$. 
\end{proposition}

\begin{proof}[Proof of Proposition~\ref{prop:compar}] 
In view of the recursive definition of $E^j$ 
in \eqref{eq:D^j,E^j}, this proof can be naturally done by induction in $k$. 
If $k=0$ then, in view of \eqref{eq:f^j} and \eqref{eq:^0}, there is nothing to prove. 
Suppose now that $k\in\intr1\infty$. 
Assume that $f^{(k)}\ge g^{(k)}$ on $I\cap(-\infty,z]$. 
Without loss of generality, $g=0$ (otherwise, replace $f$ by $f-g$). 
In view of \eqref{eq:f^j}, \eqref{eq:D^j,E^j}, and \eqref{eq:S}, one has $f^{(j)}=h_{\S\w}^{(j-1)}$ for all $j\in\intr1k$, where 
\begin{equation*}
	h:=D_{w_1}f^{(0)}=D_{w_1}R_{w_0}f. 
\end{equation*}
The conditions (i) $k\in\intr1\infty$, (ii) $f^{(j)}(z)=g^{(j)}(z)$ for all $j\in\intr0{k-1}$, (iii) $f^{(k)}\ge g^{(k)}$ on $I\cap(-\infty,z]$, (iv) $g=0$, and (v) $f^{(j)}=h_{\S\w}^{(j-1)}$ for all $j\in\intr1k$ imply 
\begin{equation}\label{eq:f^0=0}
	f^{(0)}(z)=0, 
\end{equation}
$h_{\S\w}^{(i)}(z)=0$ for all $i\in\intr0{k-2}$, and $h_{\S\w}^{(k-1)}\ge0$ on $I\cap(-\infty,z]$. 
So, by induction, $(-1)^{k-1} h\ge0$ on $I\cap(-\infty,z]$. 
Therefore, in view of \eqref{eq:D_w} and \eqref{eq:f^0=0}, 
\begin{equation*}
	\text{$(-1)^k f^{(0)}(x)=-\int_z^x(-1)^{k-1} D_{w_1}f^{(0)}(u)\d u=-\int_z^x(-1)^{k-1} h(u)\d u\ge0$}
\end{equation*}
for all $x\in I\cap(-\infty,z]$.  
Thus, the part of Proposition~\ref{prop:compar} concerning the interval $I\cap(-\infty,z]$ is proved. 
The part concerning the interval $I\cap[z,\infty)$ is proved quite similarly. 
\end{proof}

\section{\texorpdfstring{$\w$}{ww}-polynomials}\label{poly}

Take any $k\in\intr{-1}\infty$. 

\subsection{\texorpdfstring{$\w$}{ww}-polynomials: basic definitions}\label{poly-defs}

Let us say that a function $p$ is a \emph{$\w$-polynomial of degree $\le k$} (on $I$) if 
$p\in\L^{k+1}$ and $p^{(k+1)}=0$. 
Let us denote the set of all $\w$-polynomials of degree $\le k$ by $\PP^{\le k}$ or, in detailed notation, by $\PP_\w^{\le k}$. 
In particular, $\PP^{\le-1}=\{0\}$. Also, by \eqref{eq:f^j}, for any $k\in\intr0\infty$ 
\begin{equation*}
	\PP^{\le k}=\{p\in\L^k\colon p\in\L^k\text{ and $p^{(k)}$ is a constant}\}. 
\end{equation*}
%
Let 
\begin{equation}\label{eq:P_+}
	\PP_+^{\le k}:=\PP_{\w;+}^{\le k}:=\{p\in\PP^{\le k}\colon p^{(k)}\ge0\}. 
\end{equation}
In particular, $\PP_+^{\le-1}=\PP^{\le-1}=\{0\}$, $\PP^{\le0}=\{cw_0\colon c\in\R\}$, and $\PP_+^{\le0}=\{cw_0\colon c\in[0,\infty)\}$.  
Let then define the set of all $\w$-polynomials of degree $k$ as 
\begin{equation*}
	\text{$\PP^k:=\PP_\w^k:=\PP^{\le k}\setminus\PP^{\le k-1}$ for $k\in\intr0\infty$, with  $\PP^{-1}:=\PP^{\le-1}=\{0\}$. }
\end{equation*}
So, for any $k\in\intr0\infty$ 
\begin{equation*}
	\PP^k=\{p\in\L^k\colon p\in\L^k\text{ and $p^{(k)}$ is a nonzero constant}\}. 
\end{equation*}

In the unit-gauges case, the sets $\PP^k$ and $\PP^{\le k}$ coincide with the sets of usual polynomial functions on $I$ of degree $k$ and of degree $\le k$, respectively. 

\subsection{\texorpdfstring{$\w$}{ww}-polynomials: an interpolation/tangency property}\label{poly-tangency}

The following interpolation/tangency property of the $\w$-polynomials is an extension of the corresponding property of the usual polynomials. 

\begin{proposition}\label{prop:interp}
  For each $z\in I$ and each $(c_0,\dots,c_k)\in\R^{k+1}$ there is a unique $\w$-polynomial $p\in\PP^{\le k}$ such that $p^{(j)}(z)=c_j$ for all $j\in\intr0k$; moreover, this $\w$-polynomial $p$ is in $\LD_{w_0}$ and locally bounded. 
\end{proposition}

\begin{remark}\label{rem:p}
In particular, Proposition~\ref{prop:interp} implies that any $\w$-polynomial is in $\LD_{w_0}$ and locally bounded 
--  because, obviously, for any $p\in\PP^{\le k}$ and any $z\in I$ there is some finite sequence $(c_0,\dots,c_k)\in\R^{k+1}$ such that $p^{(j)}(z)=c_j$ for all $j\in\intr0k$.  
\end{remark}

\begin{proof}[Proof of Proposition~\ref{prop:interp}] 
The proof is naturally done by induction in $k$. 
If $k=-1$, there is almost nothing to prove, because then the set $\intr0k$ is empty and the set $\PP_\w^{\le k}$ is a singleton one, consisting of just one $\w$-polynomial, $0$, which is obviously in $\LD_{w_0}$ and locally bounded. 
Suppose now that $k\in\intr0\infty$. Then 
(cf.\ the proof of Proposition~\ref{prop:compar}) 
the condition that $p\in\PP_\w^{\le k}$ \,\&\, $p_\w^{(j)}(z)=c_j$ for all $j\in\intr0k$ can be rewritten as $p(z)=c_0w_0(z)$ \,\&\, $q\in\PP_{\S\w}^{\le k-1}$ \,\&\, $q_{\S\w}^{(i)}(z)=c_{i+1}$ for all $i\in\intr0{k-1}$, where 
$q:=D_{w_1}p^{(0)}=D_{w_1}R_{w_0}p$. 
By induction, the condition $q_{\S\w}^{(i)}(z)=c_{i+1}$ for all $i\in\intr0{k-1}$ determines a unique $\S\w$-polynomial $q\in\PP_{\S\w}^{\le k-1}$, and this $q$ is in $\LD_{w_1}$ and locally bounded. 
It remains to note that the conditions $D_{w_1}R_{w_0}p=q$ and $p(z)=c_0w_0(z)$ imply that $p(x)=w_0(x)\big(c_0+\int_z^x q(u)\d u\big)$ for all $x\in I$ and thus 
determine a unique $p\in\PP_\w^{\le k}$; moreover, this $p$ is in $\LD_{w_0}$ \big(since $c_0+\int_z^x q(u)\d u$ is continuous in $x$\big) 
and locally bounded 
(since both $w_0$ and $q$ are so). 
\end{proof}

\subsection{A chain of \texorpdfstring{$\w$}{ww}-polynomials vanishing at a point}\label{poly-vanish1}

Take any 
\begin{equation}\label{eq:t in}
	t\in\{a\}\cup I\setminus\{b\}=[a,b). 
\end{equation}
For $j$ and $m$ (in $\intr0\infty$\,) such that $j\le m$, define the functions $p_{t;j,m}\colon I\to(-\infty,\infty]$ recursively by the conditions
\begin{equation}\label{eq:p_tjm}
\begin{aligned}
	p_{t;m,m}(x_m)&=w_m(x_m)\quad\text{for all }x_m\in I; \\ 
	p_{t;j,m}(x_j)&=w_j(x_j)\int_{t+}^{x_j}\d x_{j+1}\,p_{t;j+1,m}(x_{j+1}) 
	\\
 	&
 	\text{\qquad\qquad\qquad for all }x_j\in I\text{\quad if }j<m. 
\end{aligned}	
\end{equation}
In the case when, for a given triple $(t,j,m)$, one has $p_{t;j,m}(x_j)\in\R$ for all $x_j\in I$, let us 
identify $p_{t;j,m}$ with the function whose graph is the same as that of $p_{t;j,m}$ but the codomain is 
$\R$.  
 
Consider first the case when $t\in I$. Then, by the local boundedness of the functions $w_0,w_1,\dots$, the functions $p_{t;j,m}$ are real-valued. Moreover, by induction, 
\begin{equation}\label{eq:p in LD}
p_{t;j,m}\in\LD_{w_j}. 	
\end{equation}
Furthermore, by \eqref{eq:D_j}, 
\begin{equation}\label{eq:Dp}
	p_{t;j+1,m}=D_j p_{t;j,m}\quad\text{and}\quad p_{t;j,m}(t)=0\quad\text{if}\quad j<m. 
\end{equation}
Hence, by \eqref{eq:f^j} and \eqref{eq:D^j,E^j}, 
$p_{t;j,m}$ is an $\S^j\w$-polynomial of degree $m-j$, satisfying the conditions  
\begin{equation}\label{eq:p}
	\big(p_{t;j,m}\big)_{\S^j\w}^{(i)}(t)=0\quad\text{for all $i\in\intr0{m-j-1}$}\quad\text{and }
	\big(p_{t;j,m}\big)_{\S^j\w}^{(m-j)}=1; 
\end{equation}
by Proposition~\ref{prop:interp}, such a polynomial is unique. 
It follows from \eqref{eq:p} that 
\begin{equation*}
	\big(p_{t;j,m}\big)_{\S^j\w}^{(i)}(t)=\ii{i=m-j}\quad\text{for all $i\in\intr0\infty$,} 
\end{equation*}
where $\ii{\cdot}$ denotes the indicator function. 
So, again by Proposition~\ref{prop:interp}, for each $k\in\intr j\infty$, the $\S^j\w$-polynomials $p_{t;j,j},\dots, p_{t;j,k}$ form a basis of the linear space $\PP_{\S^j\w}^{\le k-j}$. 
More specifically, 
each $\S^j\w$-polynomial $p$ of degree $k-j\in\intr0\infty$ can be uniquely represented by a linear combination of the basis $\S^j\w$-polynomials $p_{t;j,j},\dots,p_{t;j,k}$, as follows: 
\begin{equation}\label{eq:basis}
	p=\sum_{i=0}^{k-j}\,p_{\S^j\w}^{(i)}(t)\,p_{t;j,j+i}. 
\end{equation}


Consider now the remaining case $t\notin I$, so that, by the condition \eqref{eq:t in}, $t=a$ and $a\notin I$. Then, since $w_i\in\BB^+$ for all $i\in\intr0\infty$, the function $p_{a;j,m}$ is strictly positive on $I$ but may take the value $\infty$ at some point of the interval $I$; in such a case, it is easy to see that $p_{a;j,m}=\infty$ everywhere on $I$. 
In fact, for each pair $(j,m)\in\intr0\infty\,^2$ such that $j\le m$, one has the following dichotomy: either (i) $p_{a;j,m}=\infty$ everywhere on $I$ or (ii) $p_{a;j,m}$ is in $\PP_{\S^j\w}^{\le k-j}$ and hence locally bounded. 
Introduce the following ``finiteness'' sets for the functions $p_{a;j,m}$:
\begin{equation}\label{eq:F_}
\begin{alignedat}{2}
F&:=F_\w&&:=\big\{(j,m)\in\intr0\infty\,^2\colon j\le m\ \&\ p_{a;j,m}<\infty\big\}, \\ 
	F_{\bullet m}&:=F_{\w;\bullet m}&&:=\big\{j\colon (j,m)\in F\big\}=\big\{j\in\intr0m\colon p_{a;j,m}<\infty\big\}, \\ 
	F_{j\bullet}&:=F_{\w;j\bullet}&&:=\big\{m\colon (j,m)\in F\big\}=\big\{m\in\intr j\infty\colon p_{a;j,m}<\infty\big\}. 
\end{alignedat}		
\end{equation}
In view of \eqref{eq:p_tjm}, 
\begin{equation}\label{eq:j_m}
\text{$F_{\bullet m}=\intr{j_m}m$ for some $j_m\in\intr0m$.} 	
\end{equation}
In particular, $m\in F_{\bullet m}$ and hence $F_{\bullet m}\ne\emptyset$. Similarly, $j\in F_{j\bullet}$ and hence $F_{j\bullet}\ne\emptyset$. 
However, one has the following. 

\begin{proposition}\label{prop:M} 
Suppose that $a\notin I$. 
Then, for any $j\in\intr0\infty$ and 
any given set $M\subseteq\intr{j+1}\infty$, one can construct a sequence $\w=(w_0,w_1,\dots)$ of 
locally bounded 
functions 
in $\BB^+$
such that the set $F_{\w;j\bullet}\setminus\{j\}$ coincides with $M$. 
\end{proposition}

\begin{proof}[Proof of Proposition~\ref{prop:M}]
In the case when $a=-\infty$, for each $j\in\intr0\infty$ take some $\la_j\in\R$ and let $w_j(x_j):=\exp(\la_j x_j)$ for all $x_j\in I$. Then it is not hard to verify by induction in $m-j$ that for all $j$ and $m$ in $\intr0\infty$ such that $j<m$ and for all $x_j\in I$ 
\begin{equation}\label{eq:any M}
\begin{aligned}
	&p_{a;j,m}(x_j) \\ 
	&\qquad =\left\{
	\begin{alignedat}{2}
	&\exp\Big\{\Big(\sum_{i\in\intr jm}\la_i\Big)x_j\Big\}\Big/
	\prod_{i\in\intr j{m-1}}\,\sum_{s\in\intr{i+1}m}\la_s 
\quad&&\text{if }\la_m+
\La_{j,m}>0, \\
	&\infty 
\quad&&\text{otherwise},  
	\end{alignedat}
	\right.
\end{aligned}	
\end{equation}
where 
\begin{equation*}
	\La_{j,m}:=\min_{\rule{0pt}{9pt}i\in\intr j{m-1}}\,\sum_{s\in\intr{i+1}{m-1}
	}\la_s,  
\end{equation*}
with the usual convention that the sum of an empty family is $0$. 

In view of Remark~\ref{rem:invar}, the case when $a>-\infty$ (and hence $a\in\R\setminus I$) can be considered quite similarly. 
In this case, one may let $w_j(x_j):=(x_j-a)^{\la_j-1}$ for all $x_j\in I$. 
Then \eqref{eq:any M} holds for all $x_j\in I
$ if the exponent $\big(\sum_{i\in\intr jm}\la_i\big)x_j$ therein is replaced by $\big(-1+\sum_{i\in\intr jm}\la_i\big)\ln(x_j-a)$. 

So, in either case, whether $a=-\infty$ or $a>-\infty$, for the corresponding constructed sequence $\w$ and any $m\in\intr{j+1}\infty$ one has $m\in F_{\w;j\bullet}\setminus\{j\}$ if and only if $\la_m+
\La_{j,m}>0$. Since for any given $j$ and $m$ the real number $\La_{j,m}$ depends only on $(\la_s)_{s\in\intr{j+1}{m-1}
}$, 
one can choose $\la_m$ recursively in $m\in\intr{j+1}\infty$ so that the finiteness condition $\la_m+
\La_{j,m}>0$ in \eqref{eq:any M} be satisfied if and only if $m$ is in the prescribed subset $M$ of the set $\intr{j+1}\infty$. 
\end{proof}


%
The definitions of $F$, $F_{\bullet m}$, $F_{j\bullet}$, and $j_m$ by formulas \eqref{eq:F_} and \eqref{eq:j_m} continue to make sense even when $a\in I$, and 
\begin{align*}
	&a\in I\\
	&\quad\implies\text{$F=\big\{(j,m)\in\intr0\infty\,^2\colon j\le m\big\}$,\quad $F_{\bullet m}=\intr0m$,\quad $F_{j\bullet}=\intr j\infty$,\quad $j_m=0$.}
\end{align*}

Note also that, for each $j\in F_{\bullet m}=\intr{j_m}m$, the function $p_{a;j,m}$ is an (everywhere positive) $\S^j\w$-polynomial of degree $\le m-j$, whether $a\in I$ or not. 

In the unit-gauges case, for all $j$ and $m$ in $\intr0\infty$ one has (i) $F_{\bullet m}=\intr0m$ and $F_{j\bullet}=\intr j\infty$  if $a>-\infty$ and (ii) $F_{\bullet m}=\{m\}$ and $F_{j\bullet}=\{j\}$ if $a=-\infty$. 

Again for $j$ and $m$ in $\intr0\infty$ such that $j\le m$, define the ``positive parts'' 
of the 
functions $p_{t;j,m}$ by the formula  
\begin{align}\label{eq:p^+}
	p_{t;j,m}^+(x_j):=p_{t;j,m}(x_j)\,\ii{x_j\ge t}
\end{align}
for all $x_j\in I$. Here and subsequently, 
the convention 
\begin{equation}\label{eq:0 infty}
\infty\cdot0=0\cdot\infty=0 	
\end{equation}
is used. 

By \eqref{eq:p^+}, \eqref{eq:p_tjm}, and the positivity of the $w_j$'s, 
\begin{equation}\label{eq:p^+_{t;j,m}>0}
	p^+_{t;j,m}\ge0  
\end{equation}
for all $j$ and $m$ in $\intr0\infty$ such that $j\le m$.  
Also, 
\begin{equation}\label{eq:p^+ in LD}
	p^+_{t;j,m}\in\LD_{w_j}\text{\quad for }(j,m)\in F, 
\end{equation}
since $p_{t;j,m}\in\LD_{w_j}$ for $(j,m)\in F$. 
Moreover, one has 

\begin{lemma}\label{lem:Dp^+} 
Suppose that $j<m$ and $(j,m)\in F$.  Then
\begin{equation}\label{eq:p^+_{t;j+1,m}=}
	p^+_{t;j+1,m}=D_j\,p^+_{t;j,m}. 
\end{equation}
\end{lemma}

\begin{proof}[Proof of Lemma~\ref{lem:Dp^+}]
Take any $x_j$ and $z$ in $I$. 
In view of \eqref{eq:D_j}, \eqref{eq:D_w}, \eqref{eq:p^+}, and \eqref{eq:p^+ in LD}, it is enough to show that 
\begin{equation}\label{eq:Dp^+}
	\frac{p_{t;j,m}(x_j)}{w_j(x_j)}\,\ii{x_j\ge t}
	=\frac{p_{t;j,m}(z)}{w_j(z)}\,\ii{z\ge t}
	+\int_z^{x_j}\!\!\d x_{j+1}\,p_{t;j+1,m}(x_{j+1})\,\ii{x_{j+1}\ge t}. 
\end{equation}

In the case when $x_j\ge t$ and $z\ge t$, \eqref{eq:Dp^+} follows by \eqref{eq:D_w} and the first equality in \eqref{eq:Dp}. 

In the case when $x_j\ge t$ and $z<t$, the integral in \eqref{eq:Dp^+} equals \break $\int_t^{x_j}\!\!\d x_{j+1}\,p_{t;j+1,m}(x_{j+1})$, and so, 
\eqref{eq:Dp^+} follows by \eqref{eq:D_w} and the two equalities in \eqref{eq:Dp}. 

The case of $x_j<t$ and $z\ge t$ is quite similar to that of $x_j\ge t$ and $z<t$, as the roles of $x_j$ and $z$ are interchangeable.  

In the case when $x_j<t$ and $z<t$, \eqref{eq:Dp^+} is obvious, as each of the three indicators in \eqref{eq:Dp^+} equals $0$.  
\end{proof}


In the unit-gauges case, for $j<m$ and $x_j\in I$ 
\begin{equation}\label{eq:p_t,unit}
p_{t;j,m}(x_j)=\frac{(x_j-t)^{m-j}}{(m-j)!}\quad\text{and}\quad
p_{t;j,m}^+(x_j)=\frac{(x_j-t)_+^{m-j}}{(m-j)!}	
\end{equation}
if $t\ne-\infty$, and $p_{-\infty;j,m}=p_{-\infty;j,m}^+=\infty$; also, $p_{t;m,m}=1$ and $p_{t;m,m}^+(x_m)=\ii{x_m\ge t}$ for $x_m\in I$.

\subsection{Another chain of \texorpdfstring{$\w$}{ww}-polynomials vanishing at a point}\label{poly-vanish2}
 
Fix an arbitrary 
\begin{equation}\label{eq:z}
	z\in(a,b) 
\end{equation}
and recall \eqref{eq:F_}. 

Take any $(k,j)\in F$ and $i\in\intr0k$, and 
define the functions $p_{a,z;i:k:j}\colon I\to\R
$ by the conditions 
\begin{align}
	p_{a,z;k:k:j}(x_k)&=p_{a;k,j}(x_k)\quad\text{for all }x_k\in I; \label{eq:p_{a,z;k:k:j}} \\ 
	p_{a,z;i:k:j}(x_i)&=w_i(x_i)\int_z^{x_i}\d x_{i+1}\,p_{a,z;i+1:k:j}(x_{i+1})\quad\text{for all }x_i\in I\text{\quad if }i<k; \label{eq:p_{a,z;i:k:j}} 
\end{align}
then $p_{a,z;i:k:j}\in\PP_{\S^i\w}^{j-i}$. 

Indeed, by \eqref{eq:F_}, \eqref{eq:p_{a,z;k:k:j}}, and \eqref{eq:p_{a,z;i:k:j}}, the functions $p_{a,z;i:k:j}$ are nonnegative and finite. Also, by \eqref{eq:D_j}, 
\begin{equation}\label{eq:D_i p}
D_i\,p_{a,z;i:k:j}=p_{a,z;i+1:k:j}\text{\quad if\quad }i<k,  	
\end{equation}
and hence $D_{\S^i\w}^{k-i}\,p_{a,z;i:k:j}=p_{a,z;k:k:j}=p_{a;k,j}$. 
Moreover, \big(cf.\ \eqref{eq:p}\big) $\big(p_{a;k,j}\big)_{\S^k\w}^{(j-k)}=1$ and hence $\big(p_{a;k,j}\big)_{\S^k\w}^{(j-k+1)}=0$, 
\begin{equation}\label{eq:=1}
\big(p_{a,z;i:k:j}\big)_{\S^i\w}^{(j-i)}=1,  	
\end{equation}
and $\big(p_{a,z;i:k:j}\big)_{\S^i\w}^{(j-i+1)}=0$, which indeed yields 
\begin{equation}\label{eq:p_{a,z;i:k:j} in}
p_{a,z;i:k:j}\in\PP_{\S^i\w}^{j-i}.  	
\end{equation} 

In the unit-gauges case, for all $i\in\intr0k$ and $x_i\in I$,  
\begin{equation}\label{eq:p_a,unit,1}
	p_{a,z;i:k:j}(x_i)=\frac1{(j-i)!}\,
	\Big[(x_i-a)^{j-i}-\sum_{\ga=0}^{k-i-1}\binom{j-i}\ga\,(z-a)^{j-i-\ga}(x_i-z)^\ga\Big]
\end{equation}
if $a>-\infty$ and $j\in\intr k\infty$,  
and 
\begin{equation}\label{eq:p_a,unit,2}
	p_{a,z;i:k:k}(x_i)=\frac{(x_i-z)^{k-i}}{(k-i)!}     
\end{equation} 
whether $a=-\infty$ or $a>-\infty$. 
Recall here that the generalized polynomials $p_{a,z;i:k:j}$ were defined for $(k,j)\in F$ and $i\in\intr0k$; recall also that, in the unit-gauges case, $F=\{(j,m)\in\intr0\infty\,^2\colon j\le m\}$ if $a>-\infty$ and $F=\{(m,m)\colon m\in\intr0\infty\}$ if $a=-\infty$. 
Besides, in the case when $a>-\infty$   
\begin{equation*}
	p_{a,z;i:k:j}(x_i)\underset{z\downarrow a}\longrightarrow\frac{(x_i-a)^{j-i}}{(j-i)!} 
\end{equation*}
for all $i\in\intr0k$ and $x_i\in I$.

\section{Convex cones of generalized multiply monotone (g.m.m.) functions}\label{cones}

\subsection{Convex cones \texorpdfstring{$\H_+^{i:n}$}{H} of g.m.m.\ functions}\label{H_+}
 
Let 
$\M_+$ denote the set of all nonnegative measures $\mu$ defined on 
the Borel $\si$-algebra over $I$ such that $\mu(I\cap\{b\})=0$. 

For $j\in\intr0n$, $\mu\in\M_+$, and $x\in I$, let 
\begin{equation}\label{eq:h_{j;mu}:=}
	h_{j;\mu}(x):=h_{j:n;\mu}(x):=\smallint_I\;\mu(\d t)\,p_{t;j,n}^+(x), 
\end{equation}
so that $h_{j;\mu}(x)\in[0,\infty]$. 
Note that, if $\mu\in\M_+$ is such that $h_{j;\mu}(x)<\infty$ for all $x\in I$, then one has a function $h_{j;\mu}\colon I\to\R$. 

For each $i\in\intr0n$, let $\H_+^{i:n}$ denote the set of all functions $h\colon I\to\R$ such that (i) $h(x)=h_{i;\mu}(x)$ for some $\mu\in\M_+$ and all $x\in I$ and (ii) $h_{j;\mu}(x)<\infty$ for all $j\in\intr in$ and all $x\in I$.


\begin{lemma}\label{lem:h_{i;mu}}
Take any $i\in\intr0n$ and any $\mu\in\M_+$ such that 
$h_{i;\mu}\in\H_+^{i:n}$. 
Then $h_{i;\mu}\in\LD_{w_i}$. 
Also, the function $h_{i;\mu}/w_i$ is nondecreasing and hence 
locally bounded. 
Moreover, if $i\in\intr0{n-1}$, then $h_{i;\mu}\in\L_{w_{i+1}}$ and  
\begin{equation}\label{eq:Dh_{i;mu}}
	D_i h_{i;\mu}=h_{i+1;\mu}. 
\end{equation} 
\end{lemma}

\begin{proof}[Proof of Lemma~\ref{lem:h_{i;mu}}] 
Take any $t\in I\setminus\{b\}$. 
By \eqref{eq:p_tjm}, the function $p_{t;i,n}/w_i$ is nonnegative and nondecreasing on the interval $I\cap[t,\infty)$. 
So, by \eqref{eq:p^+}, the function $p_{t;i,n}^+/w_i$ is nonnegative and nondecreasing on the interval $I$. 
Moreover, by \eqref{eq:p in LD}, \eqref{eq:p^+}, and the condition $t\in I\setminus\{b\}$, one has $p_{t;i,n}^+\in\LD_{w_i}$. 
Next, 
\begin{equation*}
	\frac{h_{i;\mu}(x_i)}{w_i(x_i)}=\int_I\mu(\d t)\,\frac{p_{t;i,n}^+(x_i)}{w_i(x_i)}
	=\int_{I\setminus\{b\}}\mu(\d t)\,\frac{p_{t;i,n}^+(x_i)}{w_i(x_i)} 
\end{equation*}
for all $x_i\in I$, because $\mu\in\M_+$ and hence $\mu(I\cap\{b\})=0$. 
So, by dominated convergence, the condition $p_{t;i,n}^+\in\LD_{w_i}$ for $t\in I\setminus\{b\}$ implies $h_{i;\mu}\in\LD_{w_i}$.  
It also follows that $h_{i;\mu}/w_i$ is nondecreasing and hence 
locally bounded. 

Now suppose that $i\in\intr0{n-1}$. Then $h_{i+1;\mu}\in\H_+^{i+1:n}$, and so, $h_{i+1;\mu}\in\LD_{w_{i+1}}$. 
Also, by \eqref
{eq:h_{j;mu}:=}, \eqref
{eq:Dp^+}, and \eqref{eq:D_w}, 
for 
any $x_i$ and $z$ in $I$ 
\begin{equation*}
\begin{aligned}
	\int_z^{x_i}\d x_{i+1}\,h_{i+1;\mu}(x_{i+1})
	&=\int_z^{x_i}\d x_{i+1}\,\int_I\mu(\d t)\,p_{t;i+1,n}^+(x_{i+1}) \\ 
	&=\int_I\mu(\d t)\int_z^{x_i}\d x_{i+1}\,\,p_{t;i+1,n}^+(x_{i+1}) \\ 
	&=\int_I\mu(\d t)
	\Big(\frac{p_{t;i,n}^+(x_i)}{w_i(x_i)}-\frac{p_{t;i,n}^+(z)}{w_i(z)}\Big) \\ 
&	=\frac{h_{i;\mu}(x_i)}{w_i(x_i)}-\frac{h_{i;\mu}(z)}{w_i(z)},
\end{aligned}	  
\end{equation*}
which verifies \eqref{eq:Dh_{i;mu}} and thus completes the proof of Lemma~\ref{lem:h_{i;mu}}. 
\end{proof}

\subsection{Convex cones \texorpdfstring{$\F_+^{k:n}$}{F} of g.m.m.\ functions}\label{F_+}

Recall condition \eqref{eq:k le n+1}. 
Recall also \eqref{eq:f^j} and \eqref{eq:supseteq}, and introduce the class of functions 
\begin{equation}\label{eq:F}
	\F_+^{k:n}:=\F_+^{k:n}(I):=
	\big\{f\in\L^n\colon\text{$f^{(j)}$ is nondecreasing for each $j\in\intr{k-1}n$\,}\big\}. 
\end{equation}

Clearly, $\F_+^{k:n}$ is a convex cone. 
For instance, in the unit-gauges case $\F_+^{1:0}$ is the cone of all nondecreasing 
functions in $\R^I$, $\F_+^{1:1}$ is the cone of all nondecreasing continuous convex functions in $\R^I$, and $\F_+^{2:1}$ is the cone of all continuous convex functions in $\R^I$.
Also in the unit-gauges setting, special cases of 
the cones $\F_+^{k:n}$ (or similar to them) and cones in a sense dual to those cones were considered, more or less explicitly, in a number of papers, including the following: \cite{eaton1,utev-extr,pin94} for $(k,n)=(4,3)$; \cite{bent-liet02,bent-ap} (dealing with the cone $\H_+^{0:2}$; cf.\ Proposition~\ref{prop:H in F} in the present paper, below); \cite{normal} (dealing with the cone $\H_+^{0:5}$; cf.\ Theorem~\ref{th:bernoulli} in the present paper); \cite{klein-ma-priv} for $(k,n)=(2,2)$; 
\cite{asymm} for $(k,n)=(1,3)$; \cite{pin-hoeff-arxiv-reftoAIHP} for $n\in\{2,3\}$ and $k\in\intr1n$; \cite{left-arxiv} for $(k,n)\in\{(1,2),(1,3)\}$; \cite{pin-hoeff-published} (dealing with the cone $\H_+^{0:3}$). 

\begin{proposition}\label{prop:H in F}
$\H_+^{0:n}\subseteq\F_+^{k:n}$.   
\end{proposition}

\begin{proof}[Proof of Proposition~\ref{prop:H in F}]
Take any $h\in\H_+^{0:n}$, so that 
$h=h_{0;\mu}$ for some \break 
$\mu\in\M_+$. 
By \eqref
{eq:h_{j;mu}:=} and \eqref{eq:p^+_{t;j,m}>0}, $h_{i;\mu}\ge0$ for each $i\in\intr0n$. 
So, by \eqref{eq:f^j} and \eqref{eq:Dh_{i;mu}}, 
$h_{0;\mu}^{(i)}=\big(D^i h_{0;\mu}\big)/w_i=h_{i;\mu}/w_i\ge0$ for each $i\in\intr0n$. 
Hence, 
$h_{0;\mu}^{(i)}$ is nondecreasing for each $i\in\intr0{n-1}$. 
Also, in view of \eqref
{eq:h_{j;mu}:=}, \eqref{eq:p^+}, and \eqref{eq:p_tjm}, for each $x\in I$ 
\begin{equation*}
	h_{0;\mu}^{(n)}(x)=\frac{h_{n;\mu}(x)}{w_n(x)}
	=\frac1{w_n(x)}\,\int_I\;\mu(\d t)\,p_{t;n,n}^+(x)=\mu\big(I\cap(-\infty,x]\big), 
\end{equation*}
which is nondecreasing in $x$. 
Thus, by \eqref{eq:F}, $h=h_{0;\mu}\in\F_+^{k:n}$. 
\end{proof}

Important bounding properties for the functions in the class $\F_+^{k:n}$ are given by 

\begin{proposition}\label{prop:bounds}
Take any $f\in\F_+^{k:n}$ and any $z\in I$. 
\begin{enumerate}[label=\emph{(\Roman*)}]
	\item There exists a $\w$-polynomial $p\in\PP^{\le k-1}$ 
such that 
\begin{enumerate}[label=\emph{(\roman*)}]
	\item if $k$ is even, then $f\ge p$ on the interval $I$; 
	\item if $k$ is odd, then 
	\begin{enumerate}
	\item[\emph{(*)}] $f\ge p$ on the interval $I\cap[z,\infty)$; 
	\item[\emph{(**)}] $f\le p$ on the interval $I\cap(-\infty,z]$.  
\end{enumerate}
\end{enumerate}
Moreover, one may assume that this $\w$-polynomial $p$ 
depends on $f$ and $z$ only via the values of $f^{(0)}(z),\dots,f^{(k-1)}(z)$.
	\item If $k\le n$, then there exists a $\w$-polynomial $q\in\PP^{\le k}_+$ 
such that 
\begin{enumerate}[label=\emph{(\roman*)}]
	\item if $k$ is odd, then $f\ge q$ on the interval $I$; 
	\item if $k$ is even, then 
	\begin{enumerate}
	\item[\emph{(*)}] $f\ge q$ on the interval $I\cap[z,\infty)$; 
	\item[\emph{(**)}] $f\le q$ on the interval $I\cap(-\infty,z]$.  
\end{enumerate}
\end{enumerate}
Moreover, one may assume that this $\w$-polynomial $q$ 
depends on $f$ and $z$ only via the values of $f^{(0)}(z),\dots,f^{(k)}(z)$. 
\end{enumerate} 
\end{proposition}

\begin{proof}[Proof of Proposition~\ref{prop:bounds}]\ 
By Proposition~\ref{prop:interp}, there exists a unique $\w$-poly\-nomial $p\in\PP^{\le k-1}$ such that $p^{(i)}(z)=f^{(i)}(z)$ for all $i\in\intr0{k-1}$; moreover, then the condition $p\in\PP^{\le k-1}$ implies that $p^{(k-1)}(x)=p^{(k-1)}(z)=f^{(k-1)}(z)$ for all $x\in I$. 
On the other hand, the condition $f\in\F_+^{k:n}$ 
implies that the function $f^{(k-1)}$ is nondecreasing. 
Therefore,  
$f^{(k-1)}\ge p^{(k-1)}$ on the interval $I\cap[z,\infty)$, and  
$f^{(k-1)}\le p^{(k-1)}$ on the interval $I\cap(-\infty,z]$. 
To complete the proof of part~(I) of Proposition~\ref{prop:bounds}, it remains to recall Proposition~\ref{prop:compar}. 

Part~(II) of Proposition~\ref{prop:bounds} is proved similarly, by letting $q$ be the unique $\w$-poly\-nomial in $\PP^{\le k}$ such that $p^{(i)}(z)=f^{(i)}(z)$ for all $i\in\intr0k$. 
Here the additional condition $k\le n$ \big(together with 
the condition $f\in\F_+^{k:n}$\big) implies, in view of \eqref{eq:f^j}, that the function $f^{(k)}$ is nonnegative, and so, $q^{(k)}(x)=q^{(k)}(z)=f^{(k)}(z)\ge0$ for all $x\in I$.  
Therefore and because $q\in\PP^{\le k}$, it follows that 
$q\in\PP_+^{\le k}$. 
Moreover, since $f^{(k)}$ is nondecreasing, it follows that $f^{(k)}\ge q^{(k)}$ on the interval $I\cap[z,\infty)$, and  
$f^{(k)}\le q^{(k)}$ on the interval $I\cap(-\infty,z]$. 
\end{proof}

\subsection{Generalized Taylor expansion at the left endpoint \texorpdfstring{$a$}{a} of the interval \texorpdfstring{$I$}{I} of the generalized derivatives \texorpdfstring{$f^{(j)}$}{fj} for \texorpdfstring{$j\in\intr kn$}{j in} of a function \texorpdfstring{$f$}{f} in \texorpdfstring{$\F_+^{k:n}$}{F}
}\label{taylor,f^j}

Take any $f\in\F_+^{k:n}$. 
It follows by \eqref{eq:f^j} that 
\begin{equation}\label{eq:f^j(a+)}
\text{for each}\  j\in\intr kn,\quad 
\left\{ 
\begin{aligned}
	&f^{(j)}\ \text{is nonnegative and nondecreasing, and so,} \\ 
	&\text{there exists a limit}\ f^{(j)}(a+) 
	\in[0,\infty). 
\end{aligned}	
\right.
\end{equation}

Now one state the following generalized Taylor expansion. 
\begin{lemma}\label{lem:taylor-k}
For all $j\in\intr kn$
\begin{equation}\label{eq:taylor}
\begin{gathered}
	f^{(j)}\,w_j=p_j+h_j, \quad\text{where} \\ 
	p_j:=\sum_{i\in\intr jn}f^{(i)}(a+)\,p_{a;j,i}
	\quad\text{and}\quad
	h_j:=\int_{I
	}\d f^{(n)}(t)\,p_{t;j,n}^+. 
\end{gathered}
\end{equation}
\end{lemma}
The integral in \eqref{eq:taylor} is understood in the ``pointwise'' sense, 
so that $h_j(x_j)=\int_I\d f^{(n)}(t)\,p_{t;j,n}^+(x_j)$ for all $x_j\in I$;  
the latter integral exists (in $[0,\infty]$), since $p_{t;j,n}^+\ge0$ and the function $f^{(n)}$ is nondecreasing. 

 
\begin{proof}[Proof of Lemma~\ref{lem:taylor-k}]
This is done by downward induction in $j$, starting with $j=n$. Indeed, by the definitions of $p_j$ in \eqref{eq:taylor} and of $p_{t;j,m}$ in \eqref{eq:p_tjm}, 
\begin{equation*}
p_n=f^{(n)}(a+)\,w_n. 	
\end{equation*}
By the definitions of $h_j$ in \eqref{eq:taylor} and of $p_{t;j,m}^+$ in \eqref{eq:p^+}, for all $x_n\in I$ 
\begin{equation}
	h_n(x_n)=\int_{I
	}\d f^{(n)}(t)\,w_n(x_n)\ii{x_n\ge t}
	=\big(f^{(n)}(x_n)-f^{(n)}(a+)\big)\,w_n(x_n);  
\end{equation}
here we also used the fact that $f^{(n)}\in\LD$, which was noted in \eqref{eq:f^(j)more}. 
So, the equality $f^{(j)}\,w_j=p_j+h_j$ holds for $j=n$. 
Suppose now this equality holds for some $j\in\intr{k+1}n$. 
It remains to show that then this equality holds with $j-1$ instead of $j$. 
By \eqref{eq:f^j} 
and the induction assumption, 
for all $x_{j-1}\in I$ 
\begin{align}\label{eq:=J1+J2}
	f^{(j-1)}(x_{j-1})-f^{(j-1)}(a+)  
	=\int_{a+}^{x_{j-1}}\d x_j\,f^{(j)}(x_j)\,w_j(x_j)  = J_1(x_{j-1})+J_2(x_{j-1}), 
\end{align}
where 
\begin{align}
&\begin{aligned}
	J_1(x_{j-1})&:=\int_{a+}^{x_{j-1}}\d x_j\,p_j(x_j) \\ 
	&=\sum_{i\in\intr jn}f^{(i)}(a+)\,\int_{a+}^{x_{j-1}}\d x_j\,p_{a;j,i}(x_j) \\ 
	&=\sum_{i\in\intr jn}f^{(i)}(a+)\,\frac{p_{a;j-1,i}(x_{j-1})}{w_{j-1}(x_{j-1})}
	=\frac{p_{j-1}(x_{j-1})}{w_{j-1}(x_{j-1})} -f^{(j-1)}(a+)    
\end{aligned} \label{eq:J_1}
\\ 
\intertext{by \eqref{eq:taylor} and \eqref{eq:p_tjm}, whereas}  
&\begin{aligned}
	J_2(x_{j-1})&:=\int_{a+}^{x_{j-1}}\d x_j\,h_j(x_j) \\ 
	&=\int_{a+}^{x_{j-1}}\d x_j\,\int_{I
	}\d f^{(n)}(t)\,p^+_{t;j,n}(x_j) \\ 
	&=\int_{I
	}\d f^{(n)}(t)\,\int_{a+}^{x_{j-1}}\d x_j\,p_{t;j,n}(x_j)\,\ii{x_j\ge t} \\ 
	&=\int_{I
	}\d f^{(n)}(t)\,\int_{t+}^{x_{j-1}}\d x_j\,p_{t;j,n}(x_j)\,\ii{x_{j-1}\ge t} \\ 
	&=\int_{I
	}\d f^{(n)}(t)\,\frac{p_{t;j-1,n}(x_{j-1})}{w_{j-1}(x_{j-1})}\,\ii{x_{j-1}\ge t} \\ 
	&=\int_{I
	}\d f^{(n)}(t)\,\frac{p^+_{t;j-1,n}(x_{j-1})}{w_{j-1}(x_{j-1})}
	=\frac{h_{j-1}(x_{j-1})}{w_{j-1}(x_{j-1})}   
\end{aligned}	\label{eq:J_2} 
\end{align}
by \eqref{eq:taylor}, \eqref{eq:p^+_{t;j,m}>0}, the Fubini theorem, \eqref{eq:p^+}, \eqref{eq:p_tjm}, again \eqref{eq:p^+}, and again \eqref{eq:taylor}. 
Now \eqref{eq:=J1+J2}, \eqref{eq:J_1}, and \eqref{eq:J_2} indeed yield $f^{(j-1)}w^{(j-1)}=p^{(j-1)}+h^{(j-1)}$. 
\end{proof}

Take any $j\in\intr kn$, as in Lemma~\ref{lem:taylor-k}.  
Since 
$p_{a;j,i}\ge0$ for $i\in\intr jn$,  
it follows from Lemma~\ref{lem:taylor-k} and \eqref{eq:f^j(a+)} that the values of the function $h_j$ 
are all in $[0,\infty)$, i.e., are finite and nonnegative. 
It also follows, in view of \eqref{eq:F_}, 
that necessarily 
\begin{equation}\label{eq:f^j(-infty)=0}
	f^{(i)}(a+)=0\quad\text{for all}\quad i\in\intr jn\,\setminus F_{j\bullet}
	=\intr jn\,\setminus F_{\bullet n}. 
\end{equation}
Moreover, in view of \eqref{eq:taylor} and \eqref{eq:0 infty}, 
\begin{equation}\label{eq:p_k=}
	p_j=\sum_{i\in F_{j,n}}f^{(i)}(a+)\,p_{a;j,i},  
\end{equation}
where 
\begin{equation}\label{eq:F_kn}
F_{i,n}:=F_{\w;i,n}:=\intr in\,\cap F_{i\bullet}=\intr in\,\cap F_{\bullet n}=\intr{j_n\vee i}n    	
\end{equation}
for all $i\in\intr0\infty$,  
with $j_n$ defined according to \eqref{eq:j_m}. 
 
In the unit-gauges case, 
\begin{equation}\label{eq:F_kn,unit}
	F_{i,n}=
	\left\{
	\begin{aligned}
	\intr in&\text{\quad if }a>-\infty \text{ and }i\le n, \\
	\{i\}&\text{\quad if }a=-\infty \text{ and }i\le n,  \\  
	\emptyset&\text{\quad if }
	i>n,   
	\end{aligned}
	\right.
\end{equation}
and
\eqref{eq:taylor} becomes the almost usual Taylor expansion \big(of the function $f^{(j)}$ ``at the point $a+$''\big) given by the formula 
\begin{equation}\label{eq:taylor-unit}
	f^{(j)}(x_j)=\sum_{i\in\intr jn}f^{(i)}(a+)\,\frac{(x_j-a)^{i-j}}{(i-j)!}
	+\int_{I
	}\d f^{(n)}(t)\frac{(x_j-t)_+^{n-j}}{(n-j)!}    
\end{equation}
for $j\in\intr kn$ and $x_j\in I$. Here in the case when $a=-\infty$ one necessarily has $f^{(i)}(a+)=0$ for all $i\in\intr{j+1}n$ \big(cf.\ \eqref{eq:f^j(-infty)=0}\big), and then the sum in \eqref{eq:taylor-unit} reduces simply to $f^{(j)}(-\infty)$. 
For simplicity, we write $g(-\infty)$ in place of $g\big((-\infty)+\!\big)$, for any function $g$. 

Note that the set $\intr kn$ is empty if $k=n+1$, and then \eqref{eq:f^j(a+)}, Lemma~\ref{lem:taylor-k}, and \eqref{eq:taylor-unit} become vacuous. However, the definition of $p_j$ in \eqref{eq:taylor} and the expression of $p_j$ in \eqref{eq:p_k=} make sense even for $j=n+1$, 
if one uses the standard convention that the sum of any empty family is $0$, so that 
\begin{equation}\label{p_{n+1}=0}
	p_{n+1}=0. 
\end{equation}

\subsection{Truncation of the generalized Taylor expansion at the point \texorpdfstring{$a$}{a} of the generalized derivative \texorpdfstring{$f^{(k)}$}{fk} of a function \texorpdfstring{$f$}{f} in \texorpdfstring{$\F_+^{k:n}$}{F}
}\label{trunc-taylor,f^k} 

Recall \eqref{eq:z}. 
Take then any 
\begin{equation}\label{eq:y}
	y\in(a,z], 
\end{equation}
recall \eqref{eq:F_kn}, 
and introduce the function 
\begin{equation}\label{eq:tg_y}
	\tg_y:=\tg_{z,y}:=
	\sum_{j\in F_{k,n}}f^{(j)}(a+)\,p_{a,z;0:k:j}+h_{0,y}, 
\end{equation}
where $p_{a,z;0:k:j}$ is understood according to \eqref{eq:p_{a,z;k:k:j}}--\eqref{eq:p_{a,z;i:k:j}} and 
\begin{equation}\label{eq:h_i,y}
h_{i,y}:=\int_{I\cap[y,\infty)}\d f^{(n)}(t)\,p_{t;i,n}^+      
\end{equation}
for $i\in\intr0n$. 
The latter integral \big(understood in the ``pointwise'' sense, similarly to the integral expressing $h_j$ in \eqref{eq:taylor}\big) exists (in $[0,\infty]$), again because $p_{t;j,n}^+\ge0$ and the function $f^{(n)}$ is nondecreasing. 
In fact, 
$h_{i,y}(x_i)=\int_{I\cap[y,\infty)}\d f^{(n)}(t)\,p_{t;i,n}^+(x_i)\break 
=\int_{[y,y\vee x_i]}\d f^{(n)}(t)\,p_{t;i,n}(x_i)<\infty$ for all $x_i\in I$ -- because, by \eqref{eq:p_tjm} and the local boundedness of the functions $w_0,w_1,\dots$, the expression $p_{t;i,n}(x_i)$ is locally bounded in $t\in I$ for each $x_i\in I$ \big(actually, $p_{t;i,n}(x_i)$ is locally bounded in $(t,x_i)\in I^2$\big). 
So, 
in view of \eqref{eq:h_{j;mu}:=}, 
\begin{equation}\label{eq:h in}
	h_{i,y}=h_{i;\mu_{n,y}}\in\H_+^{i:n}, 
\end{equation}
where the measure $\mu_{n,y}\in\M_+$ is defined by the condition that $\mu_{n,y}\big(I\cap(-\infty,x]\big)\break =f^{(n)}(x\vee y)-f^{(n)}(y)$ for all $x\in I$; 
note here that $\mu_{n,y}(I\cap\{b\})=\mu_{n,y}(\{b\})\break
=f^{(n)}(b)-f^{(n)}(b-)=0$ if $b\in I$, since $f^{(n)}\in\LD$ by \eqref{eq:f^(j)more}, and trivially $\mu_{n,y}(I\cap\{b\})=\mu_{n,y}(\emptyset)=0$ if $b\notin I$. 
So, in either case, $\mu_{n,y}(I\cap\{b\})=0$. 

Moreover, Lemma~\ref{lem:h_{i;mu}} immediately yields 

\begin{lemma}\label{lem:h_{i,y}}
Take any $i\in\intr0n$. Then $h_{i,y}\in\LD_{w_i}$. 
Also, the function $h_{i;y}/w_i$ is nondecreasing and hence 
locally bounded. 
Moreover, if $i\in\intr0{n-1}$, then $h_{i,y}\in\L_{w_{i+1}}$ and 
\begin{equation}\label{eq:Dh_{i,y}}
	D_i h_{i,y}=h_{i+1,y}. 
\end{equation} 
\end{lemma}

Now combine \eqref{eq:tg_y}, \eqref{eq:f^j}, \eqref{eq:D^j,E^j}, \eqref{eq:D_i p}, \eqref{eq:p_{a,z;k:k:j}}, \eqref{eq:p_k=}, and \eqref{eq:Dh_{i,y}} to conclude that   
\begin{equation}\label{eq:taylor_y}
\begin{gathered}
	(\tg_y)^{(k)}\,w_k=p_k+h_{k,y}     
\end{gathered}
\end{equation}
\big(here one may want to recall that in the case when $k=n+1$ one has $\tg_y=h_{0,y}$ and, by \eqref{p_{n+1}=0}, $p_k=0$\big). 
Similarly \big(but using \eqref{eq:p_{a,z;i:k:j}} instead of \eqref{eq:p_{a,z;k:k:j}}\big), one can  also observe that 
\begin{equation}\label{eq:tg^i=h^i}
(\tg_y)^{(i)}(z)=h_{i,y}(z)\quad\text{for all }i\in\intr0{k-1},   	
\end{equation}
since $p_{a,z;i:k:j}(z)=0$ if $i<k$. 

Recalling again that $f^{(n)}\in\LD$, one has $\int_{I\cap\{a\}}\d f^{(n)}(t)\,p_{t;i,n}^+=0$. 
So, on comparing \eqref{eq:taylor_y} with \eqref{eq:taylor}, one concludes that 
\begin{equation}\label{eq:tg down}
(\tg_y)^{(k)}\near{y\downarrow a}f^{(k)}
\end{equation}
(pointwise, on $I$). 
%

\subsection{Lifting the truncated generalized Taylor expansion of \texorpdfstring{$f^{(k)}$}{fk} to an approximation \texorpdfstring{$g_y\in\PP_+^{k:n}+\H_+^{0:n}$}{gy in} of a function \texorpdfstring{$f\in\F_+^{k:n}$}{f in}
}\label{lift}

In accordance with Proposition~\ref{prop:interp}, let $q_{k;z,y}$ be the unique $\w$-polynomial  in $\PP^{\le k-1}$ such that 
\begin{equation*}
	\text{$(q_{k;z,y})^{(i)}(z)=(f-\tg_y)^{(i)}(z)$\quad for all $i\in\intr0{k-1}$.} 
\end{equation*}

Let now 
\begin{equation}\label{eq:g_y}
	g_y=g_{z,y}:=q_{k;z,y}+\tg_y. 
\end{equation}
Then 
\begin{equation}\label{eq:g_y^j}
	\text{$(g_y)^{(i)}(z)=f^{(i)}(z)$ for all $i\in\intr0{k-1}$}
\end{equation}
and 
\begin{equation}\label{eq:g^k=tg^k}
(g_y)^{(k)}=(\tg_y)^{(k)}, 	
\end{equation}
so that, by \eqref{eq:tg down}, 
\begin{equation}\label{eq:g down}
	(g_y)^{(k)}\near{y\downarrow a}f^{(k)}. 
\end{equation}
In view of 
\eqref{eq:tg_y}, 
one can rewrite \eqref{eq:g_y} as 
\begin{equation}\label{eq:g_y=}
	g_y=P_{z,y}+R_{z,y},  
\end{equation}
where 
\begin{equation}\label{eq:P=,R=}
	P_{z,y}:=q_{k;z,y}+\sum_{j\in F_{k,n}}f^{(j)}(a+)\,p_{a,z;0:k:j}\quad\text{and}\quad
	R_{z,y}:=h_{0,y}. 
\end{equation} 

Take any $j\in F_{k,n}\subseteq\intr kn$. 
By \eqref{eq:D^j,E^j}, \eqref{eq:D_i p}, and \eqref{eq:p_{a,z;k:k:j}}, $D^k\,p_{a,z;0:k:j}=p_{a,z;k:k:j}=p_{a;k,j}$. 
Therefore and by \eqref{eq:f^j} and \eqref{eq:Dp}, for each $s\in\intr kj$ one has 
$p_{a,z;0:k:j}^{(s)}=p_{a;s,j}/w_s$, which is nonnegative and nondecreasing, by \eqref{eq:p_tjm}. 
So, $p_{a,z;0:k:j}^{(s)}$ is nondecreasing
for each $s\in\intr{k-1}j$. 
Also, by \eqref{eq:=1}, $p_{a,z;0:k:j}^{(s)}=0$ 
for each $s\in\intr{j+1}n$. 
We conclude that $p_{a,z;0:k:j}^{(s)}$ is nondecreasing
for each $s\in\intr{k-1}n$. 
Also, by \eqref{eq:p_{a,z;i:k:j} in}, $p_{a,z;0:k:j}\in\PP^j\subseteq\PP^{\le n}$. 
%
%
So, 
\begin{equation}\label{eq:p_a in}
	p_{a,z;0:k:j}\in\PP_+^{k:n}:=\PP^{\le n}\,\cap\,\F_+^{k:n}
	\quad\text{for }j\in F_{k,n}.  
\end{equation}
Hence, by the condition $q_{k;z,y}\in\PP^{\le k-1}$, 
\begin{equation}\label{eq:P in}
	P_{z,y}\in\PP_+^{k:n}.  
\end{equation}
Thus, \eqref{eq:g_y=} may be considered as a Taylor-type expansion of the function $g_y$ \big(which latter is in turn an approximation to $f$, as seen from \eqref{eq:betw} below\big); at that, $R_{z,y}$ may be considered the remainder term, which vanishes when the function $f^{(n)}$ is constant on the interval $I\cap[y,\infty)$. 
In view of \eqref{eq:P=,R=}, \eqref{eq:h in}, and Proposition~\ref{prop:H in F}, 
\begin{equation}\label{eq:R in}
	R_{z,y}\in\H_+^{0:n}
	\subseteq\F_+^{k:n};  
\end{equation}
It follows from \eqref{eq:g_y=}, \eqref{eq:P in},  and \eqref{eq:R in} that 
\begin{equation}\label{eq:g_y in}
	g_y\in\PP_+^{k:n}+\H_+^{0:n}\subseteq\F_+^{k:n}. 
\end{equation}

In view of \eqref{eq:F_kn,unit}, in the unit-gauges case with $a=-\infty$ and $k\le n$, the summands $P_{z,y}$ and $R_{z,y}$ in \eqref{eq:g_y=} 
are as follows: for all $x_0\in I$,  
\begin{align}
&
	P_{z,y}(x_0)=\sum_{i=0}^{k-1}c_{i;z,y}\frac{(x_0-z)^i}{i!}
	+f^{(k)}(-\infty)\,\frac{(x_0-z)^k}{k!},  \quad\text{with } \label{eq:P_zy} \\ 
	&\qquad\qquad\qquad\ \begin{aligned}c_{i;z,y}:=&f^{(i)}(z)-(\tg_y)^{(i)}(z) \\ 
	=&f^{(i)}(z)-h_{i,y}(z)\\ 
	=&f^{(i)}(z)-\int_{I\cap[y,\infty)}\d f^{(n)}(x_n)\,\frac{(z-x_n)_+^{n-i}}{(n-i)!},   
\end{aligned}
 \notag \\
\intertext{and}
	&R_{z,y}(x_0)=\int_{I\cap[y,\infty)}\d f^{(n)}(x_n)\,\frac{(x_0-x_n)_+^n}{n!};  \label{eq:R_zy}
\end{align}
here the second expression for $c_{i;z,y}$ is obtained by \eqref{eq:tg^i=h^i}; if $a>-\infty$ or $k=n+1$, the expression for $P_{z,y}$ is simpler than the one in \eqref{eq:P_zy}. 

By \eqref{eq:P in}, 
Remark~\ref{rem:p}, Lemma~\ref{lem:h_{i,y}}, and the 
continuity of the $w_i$'s, the functions $P_{z,y}$ and $h_{0,y}$ are locally bounded, for each 
$y$. 
Similarly, by \eqref{eq:g^k=tg^k} and \eqref{eq:taylor_y}, $
(g_y)^{(k)}\,w_k$ is locally bounded, for each $y$. In particular, $D_\w^k g_z$ is locally bounded. 
Now by 
Proposition~\ref{prop:compar} and monotone convergence one obtains the following approximative representation of any function $f\in F_+^{k:n}$ by mixtures of $\w$-polynomials and ``positive parts'' thereof. 

\begin{theorem}\label{th:g_y->f} 
For any $f\in F_+^{k:n}$ and $g_y$ as in \eqref{eq:g_y=}--\eqref{eq:P=,R=}, 
\begin{equation}\label{eq:betw}
g_y\near{y\downarrow a}f\ \text{on $I\cap[z,\infty)$\quad and\quad } 	
(-1)^k (f-g_y)\sear{y\downarrow a}0\ \text{on $I\cap(-\infty,z]$}. 
\end{equation}
\end{theorem}

\begin{remark}\label{rem:left}
In view of \eqref{eq:betw} and \eqref{eq:g_y in}, the cone $\F_+^{k:n}$ of functions $f$ on $I$ can be viewed as the closure, in a certain topology, of the cone $\PP_+^{k:n}+\H_+^{0:n}$. 
One may therefore 
ask whether these two cones are the same, that is, whether $\F_+^{k:n}=\PP_+^{k:n}+\H_+^{0:n}$. 
However, in general this is not the case, for any $n$ and $k$ as in \eqref{eq:k le n+1}. 
For example, in the unit-gauges setting let $I=\R$ 
and 
\begin{equation}\label{eq:ex}
	f(x):=g(x)\ii{x\le0}+p(x)\ii{x>0} 
\end{equation}
for all $x\in\R$, 
where $g(x):=(-1)^k(1-x)^{k-1/2}$ and $p(x):=\sum_{i=0}^{n+1}g^{(i)}(0)x^i/i!$. 
Then $f^{(j)}(x)=g^{(j)}(x)\ii{x\le0}+\sum_{i=j}^{n+1}g^{(i)}(0)x^{i-j}\ii{x>0}/(i-j)!>0$ for all $j\in\intr k{n+1}$ and $x\in\R$, and hence $f\in\F_+^{k:n}$. 
On the other hand, $f\notin\PP_+^{k:n}+\H_+^{0:n}$. 
Indeed, take any $q\in\PP_+^{k:n}$ and $h\in\H_+^{0:n}$. Then, 
by \cite[Lemma~2 on page~619 and formula~(1) on page~606]{asymm}, 
$h(-\infty)=0$; so, for $x\to-\infty$, either $q(x)+h(x)\sim c|x|^k$ for some $c\in(0,\infty)$ or $q(x)+h(x)=O(|x|^{k-1})$ -- depending on whether the degree of the polynomial $q$ is $n$ or $<n$, whereas $|f(x)|\sim|x|^{k-1/2}$.   

Quite similarly, one can show that $\F_+^{k:n}\ne\PP_+^{k:n}+\H_+^{0:n}$ for any interval $I$ with the left endpoint $a=-\infty$, again in the unit-gauges case and again for any $n$ and $k$ as in \eqref{eq:k le n+1}. 
Moreover, in view of Remark~\ref{rem:invar}, it is easy to see that $\F_+^{k:n}\ne\PP_+^{k:n}+\H_+^{0:n}$ for any given nonzero-length interval $I$ and any $n$ and $k$ as in \eqref{eq:k le n+1}, with an appropriate choice of gauge functions $w_0,w_1,\dots$. 

The idea of construction \eqref{eq:ex} comes from \cite{asymm}; cf.\ Propositions~ 1 and 2 therein. As mentioned before, in \cite{asymm} the special case with $n=3$ and $k=1$ was considered. 
\end{remark}

\section{Convex cone dual to \texorpdfstring{$\F_+^{k:n}$}{F} }\label{dual}
Let us recall that condition \eqref{eq:k le n+1} continues to hold in this section.  
Here we shall define and completely characterize the convex cone dual to any set $\G$ such that 
\begin{equation}\label{eq:G}
	\PP_+^{k:n}\cup\PP^+_{0:n}\subseteq\G\subseteq\F_+^{k:n},   
\end{equation}
where $\PP_+^{k:n}$ is as defined in \eqref{eq:p_a in} and 
\begin{equation}\label{eq:PP^+_{0:n}}
	\PP^+_{0:n}:=\{p_{t;0,n}^+\colon t\in I\}, 
\end{equation}
with $p_{t;0,n}^+$ defined according to \eqref{eq:p^+}. 
 
Note that  
\begin{gather}
	\PP^{\le k-1}\subseteq\PP_+^{\le k}, \label{eq:in in,1} \\ 
	\PP^{\le k-1}\subseteq\G, \label{eq:in in,2} \\ 
	k\le n\implies\PP_+^{\le k}\subseteq\G. \label{eq:in in,3} 
\end{gather}
Indeed, by the definition, if $p\in\PP^{\le k-1}$, then $p^{(k)}=0$ and $p^{(k-1)}$ is a constant, whence $p\in\PP_+^{\le k}$, by \eqref{eq:P_+}, and $p\in\PP^{\le n}\,\cap\,\F_+^{k:n}=\PP_+^{k:n}\subseteq\G$, by \eqref{eq:k le n+1}, \eqref{eq:F}, \eqref{eq:p_a in}, and \eqref{eq:G}. This yields \eqref{eq:in in,1} and \eqref{eq:in in,2}. 

If now $k\le n$ and $p\in\PP_+^{\le k}$, then obviously $p\in\PP^{\le n}$, and also $p^{(k)}\ge0$ and $p^{(k+1)}=0$, whence, in view of \eqref{eq:f^j}, $p^{(k-1)}$ is nondecreasing and $p^{(j)}$ is constant for each $j\in\intr kn$, so that, by \eqref{eq:F}, $p\in\F_+^{k:n}$. Thus, recalling again the definition of $\PP_+^{k:n}$ in \eqref{eq:p_a in} and the first set inclusion in \eqref{eq:G}, one obtains \eqref{eq:in in,3}. 


Also, one has 
\begin{proposition}\label{prop:p^+}
$\PP^+_{0:n}\subseteq\F_+^{1:n}\subseteq\F_+^{k:n}$. 
\end{proposition}
\begin{proof}[Proof of Proposition~\ref{prop:p^+}]
Take any $t\in I$. In view of \eqref{eq:p^+_{t;j+1,m}=}, \eqref{eq:f^j}, and \eqref{eq:p^+_{t;j,m}>0}, $\big(p_{t;0,n}^+\big)^{(j)}=p_{t;j,n}^+/w_j\ge0$ for all $j\in\intr0n$. 
In particular, $\big(p_{t;0,n}^+\big)^{(n)}(x_0)=\ii{x_0\ge t}$ is obviously nondecreasing in $x_0\in I$. 
It also follows that for each $j\in\intr0{n-1}$ one has $(p_{t;0,n}^+)^{(j+1)}w_{j+1}=p_{t;j+1,n}^+\ge0$ and hence, by \eqref{eq:f^j},  $(p_{t;0,n}^+)^{(j)}$ is nondecreasing. 
So, $p_{t;0,n}^+\in\F_+^{1:n}\subseteq\F_+^{k:n}$. 
Now Proposition~\ref{prop:p^+} follows by \eqref{eq:PP^+_{0:n}}. 
\end{proof}

By \eqref
{eq:p_a in} and Proposition~\ref{prop:p^+}, there always is a set $\G$ satisfying conditions~\eqref{eq:G},  
which will be the only conditions generally imposed on $\G$ in this paper. 
In particular, the set $\G$ will not have to be convex or a cone. However, the cone dual to $\G$, to be denoted by $\hG$ and defined later in this section, will be a convex cone indeed. Moreover, it will turn out that in most cases the dual cone $\hG$ will not depend on the choice of $\G$ as long as conditions \eqref{eq:G} are satisfied -- the only exception in this regard being the case when all of the following conditions hold: 
\begin{equation}\label{eq:except}
\text{$k=n+1$, $k$ is odd, and $a\notin I$.}	
\end{equation}
So, unless this exceptional case takes place, the dual cone $\hG$ will coincide with $\hF_+^{k:n}$.

\subsection{Admissible set of measures}\label{admissible}
In accordance with the general definition of the dual cone (see e.g.\ \cite[Ch.\ III, Section~5]{ekeland-temam} or \cite[page~7]{zalinescu}), it appears natural to define the cone $\hG$ dual to the set $\G$ of functions on $I$ as consisting of signed measures on the Borel $\si$-algebra -- say $\B$ --  
over $I$. 
However, we shall take a more general approach by letting $\hG$ be a set of ordered pairs $(\nu_1,\nu_2)$ of nonnegative (not necessarily finite) measures on $\B$ such that $\nu_1(f)\ge\nu_2(f)$ for all $f\in\G$. Here and subsequently, we use the common definition $\nu(f):=\int_I\,f\d\nu$ for a Borel-measurable function $f\colon I\to\R$ and a nonnegative measure $\nu$ on $\B$, if the integral exists in the extended sense, that is, if at least one of the values $\nu(f_+)$ or $\nu(f_-)$ is finite, and in such a case we let $\nu(f):=\nu(f_+)-\nu(f_-)$; as usual, $f_+:=f\vee0$ and $f_-:=(-f)_+$. 
For brevity (unless otherwise indicated), when we say that $\nu(f)$ satisfies a certain condition, it will actually mean that $\nu(f)$ \emph{exists and} satisfies that condition. E.g., if we say $\nu(f)>-\infty$, it actually means that $\nu(f)$ exists and does not equal $-\infty$ \big(which is equivalent to the statement that $\nu(f_-)<\infty$\big). 

Of course, if at least one of the nonnegative measures $\nu_1,\nu_2$ is finite, then one can introduce the signed measure $\nu:=\nu_1-\nu_2$; if, moreover, at least one of the integrals $\nu_1(f),\nu_2(f)$ is finite, then one can also let $\nu(f):=\nu_1(f)-\nu_2(f)$ and write the usual duality condition $\nu(f)\ge0$ instead of $\nu_1(f)\ge\nu_2(f)$.  
However, such additional restrictions on the finiteness of one of the measures $\nu_1,\nu_2$ or one of the integrals $\nu_1(f),\nu_2(f)$ are unnecessary for our results on the dual cone or in the relevant applications. 

Yet, in order to ensure that the dual cone $\hG$ be convex, one cannot allow two pairs $(\nu_1,\nu_2)$ and $(\rho_1,\rho_2)$ of measures to both belong to $\hG$ if $\{\nu_j(f),\rho_j(f)\}=\{\infty,-\infty\}$ for some $f\in\G$ and some $j\in\{1,2\}$ -- because in that case the integral $(\nu_j+\rho_j)(f)$ would not exist and thus the pair $(\nu_1,\nu_2)+(\rho_1,\rho_2)=(\nu_1+\rho_1,\nu_2+\rho_2)$ could not possibly belong to $\hG$. 
For this reason, only pairs $(\nu_1,\nu_2)$ of nonnegative measures such that $\nu_1(f)\wedge\nu_2(f)>-\infty$ for all $f\in\G$ will be allowed to belong to the dual cone $\hG$; such pairs of measures may be referred to as admissible.  

To formalize this approach to admissibility (which works well in the applications), let us first introduce the notation $\NN_+$ for the set of all nonnegative (not necessarily finite) measures on $\B$.   
Introduce next the set
\begin{align}\label{eq:N_+^k}  
		\NN_+(\G)
		&:=\big\{\nu\in\NN_+\colon\nu(f)>-\infty\ \text{for all }f\in\G\big\},    
\end{align}
which may be referred to as the admissible set (of nonnegative measures corresponding to the set $\G$ of functions). 
One has the following characterization of this admissible set. 

\begin{proposition}\label{prop:N_+^k}
Take any $\nu\in\NN_+$.  
\begin{enumerate}[label=\emph{(\roman*)}]
	\item If $k\le n$, then 
\begin{equation}\label{eq:iff1}
\begin{aligned}
\nu\in\NN_+(\G)&\iff\nu(p)>-\infty\text{ for all }p\in\PP_+^{\le k}.   
\end{aligned}
\end{equation}	
	\item If $k$ is even or $a\in I$, then 
\begin{equation}\label{eq:iff2}
\begin{aligned}
\nu\in\NN_+(\G)&\iff\nu(p)>-\infty\text{ for all }p\in\PP^{\le k-1} \\
  &\iff\nu(p)\in\R\text{ for all }p\in\PP^{\le k-1}
\end{aligned}
\end{equation}
\item If the exceptional case \eqref{eq:except} takes place
, then  
\begin{equation}\label{eq:iff3}
\begin{aligned}
\nu\in\NN_+\big(\F_+^{k:n}\big)\iff &
\nu(p)\in\R\text{ for all }p\in\PP^{\le k-1} \\
&\quad\&\ \nu\big((a,\ta)\big)=0\text{ for some }\ta\in I \\
\implies & \nu\in\NN_+(\G).   
\end{aligned}
\end{equation}	
\end{enumerate}
\end{proposition}


Thus \big(given the condition \eqref{eq:G}\big), the admissible set $\NN_+(\G)$ does not actually depend on the choice of  $\G$ -- except for the case \eqref{eq:except}. 
In that exceptional case, \eqref{eq:iff3} shows that the admissible set $\NN_+\big(\F_+^{k:n}\big)$ is inconveniently too small, consisting only of measures $\nu$ with support $\supp\nu$ bounded away from the left endpoint $a$ of the interval $I$; in particular, in the important case when $a=-\infty$, the set $\supp\nu$ will have to be bounded from below, which would rule out applications without such a restriction. 
Allowing $\G$ to differ from $\F_+^{k:n}$ was motivated by this inconvenience. 
Indeed, if the class $\G$ is smaller $\F_+^{k:n}$ then, in view of \eqref{eq:N_+^k}, the admissible set $\NN_+(\G)$ may turn out to be a large enough extension of the too small class $\NN_+\big(\F_+^{k:n}\big)$ of measures.  
In particular, a 
sensible choice of $\G$ in the exceptional case \eqref{eq:except} appears to be given by the formula 
\begin{equation}\label{eq:G,except}
	\G=\{f\in\F_+^{n+1:n}\colon f\ge p\text{ for some }p\in\PP^{\le n}\}, 
\end{equation}
so that the equivalences in \eqref{eq:iff2} obviously continue to hold even in the exceptional case \eqref{eq:except}. 

\begin{proof}[Proof of Proposition~\ref{prop:N_+^k}]\ \\ 
The implication $\implies$ in 
part~(i) of this proposition follows immediately by \eqref{eq:N_+^k} 
and 
\eqref{eq:in in,3}, whereas the reverse implication $\Longleftarrow$ there 
follows by 
parts~(II)(i) and (I)(i) of Proposition~\ref{prop:bounds} and 
\eqref{eq:in in,1}. 

If $k$ is even, then the first equivalence in 
part~(ii) of Proposition~\ref{prop:N_+^k} follows by \eqref{eq:N_+^k}, 
part~(I)(i) of Proposition~\ref{prop:bounds}, and \eqref
{eq:in in,2}. 
If $k$ is odd and $a\in I$, then the just mentioned equivalence 
follows by \eqref{eq:N_+^k}, 
part~(I)(ii)(*) of Proposition~\ref{prop:bounds} (with $z=a$), and \eqref
{eq:in in,2}. 
As for the second equivalence in part~(ii) of Proposition~\ref{prop:N_+^k}, it follows because $-p\in\PP^{\le k-1}$ for any $p\in\PP^{\le k-1}$. 

To complete the proof of Proposition~\ref{prop:N_+^k}, it remains to prove its part~(iii). 
To do this, suppose first $\nu\big((a,\ta)\big)=0$ for some $\ta\in I$. 
Then one can replace $\nu$ by its restriction to the Borel $\si$-algebra over the reduced interval $I\cap[\ta,\infty)$ in place of $I$ and, accordingly, replace the functions in $\G$ and the functions $w_0,w_1,\dots$ by their respective restrictions to the interval $I\cap[\ta,\infty)$, which obviously contains its left endpoint $\ta$. Now the implication $\implies$ in the last line in \eqref{eq:iff3} and, in particular, the implication $\Longleftarrow$ in the first line there follow immediately by already verified part~(ii) of Proposition~\ref{prop:N_+^k}. 

The implication $\nu\in\NN_+\big(\F_+^{k:n}\big)\implies
\nu(p)\in\R\text{ for all }p\in\PP^{\le k-1}$ in \eqref{eq:iff3} follows by 
\eqref{eq:in in,2}, \eqref{eq:G}, and the second equivalence in \eqref{eq:iff2} (which latter holds whether or not the exceptional case \eqref{eq:except} takes place). 

Thus, it remains to verify the implication $\nu\in\NN_+\big(\F_+^{k:n}\big)\implies
\nu\big((a,\ta)\big)=0\text{ for some }\ta\in I$ in \eqref{eq:iff3}. Toward this end, assume, on the contrary, that $\nu\big((a,\ta)\big)>0$ for all $\ta\in I$. 
Take any sequence $(t_i)_{i\in\N}$ in $I$ such that $t_i\to
a$ (as $i\to
\infty$). 
Then, by the assumption, $\nu\big((a,t_i)\big)>0$ for all $i\in\N$. 
Introduce now the functions $p_{t;j,m}^-\colon I\to\R$ by the formula 
\begin{align}\label{eq:p^-}
	p_{t;j,m}^-(x_j):=p_{t;j,m}(x_j)\,\ii{x_j<t}
\end{align}
for all $t$ and $x_j\in I$ (cf.\ \eqref{eq:p^+}). 
The conditions in \eqref{eq:except} that $k = n + 1$ and $k$ is odd imply that $n$ is even. 
So,    
by \eqref{eq:p^-} and \eqref{eq:p_tjm}, $p_{t;0,n}^-\ge0$ (on $I$) and $p_{t;0,n}^->0$ on $I\cap(-\infty,t)=(a,t)$, for any $t\in I$. 
Take now any $i\in\N$. Then, recalling that $\nu\big((a,t_i)\big)>0$, take any $\ga_i\in\R$ such that 
\begin{equation}\label{eq:ga_i}
	0<\ga_i\le\nu(p_{t_i;0,n}^-);  
\end{equation} 
in particular, if $\nu(p_{t_i;0,n}^-)<\infty$ one may take $\ga_i=\nu(p_{t_i;0,n}^-)$. 
Now introduce the function 
\begin{equation}\label{eq:f=sum}
	f:=-\sum_{i=1}^\infty\frac1{\ga_i}\,p_{t_i;0,n}^-
	=-\int_I\mu(\d t)\,p_{t;0,n}^-,  
\end{equation}
where 
$\mu$ is the nonnegative measure defined by the formula  
\begin{equation*}
	\mu(g):=\sum_{i=1}^\infty\frac1{\ga_i}\,g(t_i)
\end{equation*}
for all nonnegative functions $g$ on $I$. 
Since $p_{t;0,n}^-\ge0$ for all $t\in I$, it follows that $f\le0$. Moreover, in view of the condition $a\notin I$ in \eqref{eq:except}, for any $x\in I$ there is some $i_x\in\N$ such that for all $i\in\intr{i_x}\infty$ one has $t_i<x$ and hence $p_{t_i;0,n}^-(x)=0$, so that $f(x)=-\sum_{i=1}^{i_x-1}\frac1{\ga_i}\,p_{t_i;0,n}^-(x)>-\infty$. 
Therefore, $f(x)\in(-\infty,0]$ for all $x\in I$. 
Next \big(cf.\ \eqref{eq:f=sum}, 
Lemma~\ref{lem:h_{i;mu}}, and \eqref{eq:p_tjm}\big), 
\begin{equation*}
\begin{aligned}
f^{(n)}(x)&=-\frac1{w_n(x)}\,\int_I\mu(\d t)\,p_{t;n,n}^-(x) \\ 
&=-\frac1{w_n(x)}\,\int_I\mu(\d t)\,p_{t;n,n}(x)\ii{x<t} \\ 
&=-\int_I\mu(\d t)\ii{x<t}=-\mu\big(I\cap(x,\infty)\big)	
\end{aligned}
\end{equation*}
is nondecreasing in $x\in I$, so that $f\in\F_+^{n+1:n}=\F_+^{k:n}$. 
On the other hand, by \eqref{eq:f=sum} and \eqref{eq:ga_i}, 
\begin{equation*}
	\nu(f)=-\sum_{i=1}^\infty\frac1{\ga_i}\,\nu(p_{t_i;0,n}^-)\le-\sum_{i=1}^\infty1=-\infty. 
\end{equation*}
So, the 
assumption that the condition ``$\nu\big((a,\ta)\big)=0$ for some $\ta\in I$\kern1pt'' is violated has led to the conclusion that $\nu\notin\NN_+\big(\F_+^{k:n}\big)$. 
This completes the proof of part~(iii) of Proposition~\ref{prop:N_+^k} as well. 
\end{proof}

Suppose, for example, that the exceptional case \eqref{eq:except} with $n=0$ takes place in the unit-gauges setting. Then $\F_+^{k:n}=\F_+^{n+1:n}=\F_+^{1:0}$ is the set of all nondecreasing 
functions $f\colon I\to\R$. 
In this situation, with $a\notin I$, it is rather clear that, for any nonnegative measure $\nu$ on $\B$ with $\inf\supp\nu=a$, one can choose a function $f\in\F_+^{1:0}=\F_+^{k:n}$ growing so fast (from $-\infty$ up) on $I$ that $\nu(f)=-\infty$. This simple observation was the main idea behind the above proof of part~(iii) of Proposition~\ref{prop:N_+^k}. 
Again for the exceptional case \eqref{eq:except} with $n=0$ in the unit-gauges setting, the set $\G$ as in \eqref{eq:G,except} is the set of all nondecreasing 
functions in $\LD$ that are \emph{also bounded from below}, which latter appears to be a rather natural additional condition to impose on the functions in $\F_+^{1:0}$.  

The function $f$ defined by formula \eqref{eq:f=sum} in the proof of part~(iii) of Proposition~\ref{prop:N_+^k} may be considered a generalized spline of order $n$. For instance, in the unit-gauges setting with $n=2$ and $k=n+1=3$, that function $f$ will be continuously differentiable, with the graph consisting of countably many parabolic arcs.


\subsection{Dual cone}\label{dual G}
Define the dual cone $\hG$ by the formula 
\begin{equation}\label{eq:hG}
\begin{aligned}
	\hG:=&\big\{(\nu_1,\nu_2)\in\NN_+(\G)\times\NN_+(\G)\colon\nu_1(f)\ge\nu_2(f)\ \text{for all }f\in\G\big\} \\
	=&\big\{(\nu_1,\nu_2)\in\NN_+\times\NN_+\colon\nu_1(f)\ge\nu_2(f)>-\infty\ \text{for all }f\in\G\big\}. 
\end{aligned}	
\end{equation}

\begin{theorem}\label{th:dual cone} 
Take any $(\nu_1,\nu_2)\in\NN_+(\G)\times\NN_+(\G)$. Then $(\nu_1,\nu_2)\in\hG$ if and only if all of the following conditions hold: 
	\begin{enumerate}[label=\emph{(\roman*)}]
	\item $\nu_1(p)=\nu_2(p)\in\R
	$ for all $p\in\PP^{\le k-1}$;   
	\item 
	$\nu_1(p)\ge\nu_2(p)$ for all $p\in\PP_+^{k:n}\big[=\PP^{\le n}\,\cap\,\F_+^{k:n}$, by \eqref{eq:p_a in} 
	\big];  
	\item $\nu_1(p_{t;0,n}^+)\ge\nu_2(p_{t;0,n}^+)$ 
	for all $t\in I$. 
\end{enumerate}
\end{theorem} 
Thus, the verification of the condition on $(\nu_1,\nu_2)\in\NN_+(\G)\times\NN_+(\G)$ that $\nu_1(f)\ge\nu_2(f)$ for all $f\in\F_+^{k:n}$ reduces to the verification of this inequality just for certain $\w$-polynomials and their ``positive parts''. 

\begin{remark}\label{rem:ii->i}
For each $(\nu_1,\nu_2)\in\NN_+(\G)\times\NN_+(\G)$, condition (i) of Theorem~\ref{th:dual cone} follows from condition (ii) there. 
Indeed, suppose that condition~(ii) holds, and then  
take any $(\nu_1,\nu_2)\in\NN_+(\G)\times\NN_+(\G)$ and any $p\in\PP^{\le k-1}$. 
Then $\{p,-p\}\subseteq\PP^{\le k-1}
\subseteq\G$ by \eqref{eq:in in,2}, 
whence 
$\nu_1(p)\ge\nu_2(p)>-\infty$ and $-\nu_1(p)=\nu_1(-p)\ge\nu_2(-p)=-\nu_2(p)$, with $\nu_2(-p)>-\infty$, which yields $\nu_1(p)=\nu_2(p)\in\R$. 
Thus, again for each $(\nu_1,\nu_2)\in\NN_+(\G)\times\NN_+(\G)$, conditions (ii) and (iii) of Theorem~\ref{th:dual cone} already suffice for $(\nu_1,\nu_2)\in\hG$. 
Moreover, in the case when $k=n+1$, one has $\PP^{\le k-1}=\PP^{\le n}=\PP_+^{k:n}$, because $p^{(n)}$ is constant and hence nondecreasing for any $p\in\PP^{\le n}$; therefore, in this case conditions (i) and (ii) of Theorem~\ref{th:dual cone} are just equivalent to each other. 
\end{remark}

\begin{proof}[Proof of Theorem~\ref{th:dual cone}]
That condition~(ii) in Theorem~\ref{th:dual cone} is necessary for \break 
$(\nu_1,\nu_2)\in\hG$ follows immediately from \eqref{eq:hG} and \eqref{eq:G}. 
Next, condition~(i) follows from condition~(ii) by Remark~\ref{rem:ii->i}. 
The necessity of condition~(iii) in Theorem~\ref{th:dual cone} follows immediately from \eqref{eq:G} and \eqref{eq:PP^+_{0:n}}. 
Thus, the ``only if'' part of Theorem~\ref{th:dual cone} is verified. 

Let us now consider the ``if'' part of the theorem. 
Suppose that conditions (i)--(iii) of Theorem~\ref{th:dual cone} hold. 
Take any $z\in(a,b)$ and then $y\in(a,z]$, as in \eqref{eq:z} and \eqref{eq:y}.  
Take also any $f\in\G
$. Then, by \eqref{eq:G}, $f\in\F^{k:n}$, and so, the function $f^{(n)}$ is nondecreasing and hence the corresponding Lebesgue--Stieltjes measure $\d f^{(n)}$ is nonnegative. 
So, by the definition of $R_{z,y}$ in \eqref{eq:P=,R=}, \eqref
{eq:h_i,y}, condition (iii) of Theorem~\ref{th:dual cone}, the Fubini theorem, and \eqref{eq:p^+_{t;j,m}>0}, 
\begin{equation}\label{eq:R compar}
	\nu_1(R_{z,y})\ge\nu_2(R_{z,y})\ge0. 
\end{equation} 
By \eqref{eq:P in}, \eqref{eq:G}, conditions (ii) of Theorem~\ref{th:dual cone} and $(\nu_1,\nu_2)\in\NN_+(\G)\times\NN_+(\G)$, 
and \eqref{eq:N_+^k}, 
\begin{equation}\label{eq:P compar}
	\nu_1(P_{z,y})\ge\nu_2(P_{z,y})>-\infty. 
\end{equation}
It follows by \eqref{eq:g_y=}, \eqref{eq:R compar}, and \eqref{eq:P compar} 
that 
\begin{equation}\label{eq:g_y ineq}
	\nu_1(g_y)\ge\nu_2(g_y). 
\end{equation}

Also, by \eqref{eq:g_y in}, $g_z\in\F_+^{k:n}$.  
So, by part~(I) of Proposition~\ref{prop:bounds}, 
there exists a $\w$-polynomial 
$p_z\in\PP^{\le k-1}$ 
such that 
(i) if $k$ is even, then $g_z\ge p_z$ 
and (ii) if $k$ is odd, then 
$g_z\ge p_z$ on $I\cap[z,\infty)$ and 
$g_z\le p_z$ on $I\cap(-\infty,z]$.
In view of \eqref{eq:betw}, one concludes that 
\begin{enumerate}[label=(\roman*)]
	\item if $k$ is even, then $g_y\ge p_z$ and $g_y\near{y\downarrow a}f$ on the interval $I$; 
	\item if $k$ is odd, then 
	\begin{enumerate}
	\item[(*)] $g_y\ge p_z$ and $g_y\near{y\downarrow a}f$ on the interval $I\cap[z,\infty)$;
	\item[(**)] $
	\label{bounds on g_y}
	g_y\le p_z$ and $g_y\sear{y\downarrow a}f$ on the interval $I\cap(-\infty,z]$.   
\end{enumerate}
\end{enumerate}
By condition (i) of Theorem~\ref{th:dual cone}, $\nu_1(p_z)=\nu_2(p_z)\in\R$.  
Thus, by the Lebesgue monotone convergence theorem and conclusions (i) and (ii) above, 
\begin{equation}\label{eq:converg}
\int\limits_{I\cap[z,\infty)}g_y\d\nu_j\underset{y\downarrow a}\longrightarrow\int\limits_{I\cap[z,\infty)}f\d\nu_j
\quad\text{and}\quad 
\int\limits_{I\cap(-\infty,z)}g_y\d\nu_j\underset{y\downarrow a}\longrightarrow\int\limits_{I\cap(-\infty,z)}f\d\nu_j  	
\end{equation}
for $j\in\{1,2\}$. 
In view of the condition $(\nu_1,\nu_2)\in\NN_+(\G)\times\NN_+(\G)$ 
and the definition 
\eqref{eq:N_+^k}, 
$\nu_1(f)\wedge\nu_2(f)>-\infty$ or, equivalently, $\nu_1(f_-)\vee\nu_2(f_-)<\infty$, whence 
\begin{equation}\label{eq:wedge}
\int\limits_{I\cap[z,\infty)}f\d\nu_1\bigwedge\int\limits_{I\cap[z,\infty)}f\d\nu_2\bigwedge
\int\limits_{I\cap(-\infty,z)}f\d\nu_1\bigwedge\int\limits_{I\cap(-\infty,z)}f\d\nu_2>-\infty. 	
\end{equation}
It follows from \eqref{eq:converg} and \eqref{eq:wedge} that $\nu_j(g_y)=\int_{I\cap[z,\infty)}g_y\d\nu_j+\int_{I\cap(-\infty,z)}g_y\d\nu_j\break \underset{y\downarrow a}\longrightarrow\int_{I\cap[z,\infty)}f\d\nu_j+\int_{I\cap(-\infty,z)}f\d\nu_j=\nu_j(f)$, again for $j\in\{1,2\}$; condition \eqref{eq:wedge} is used here to show that 
the integrals $\int_{I\cap[z,\infty)}f\d\nu_j$ and $\int_{I\cap(-\infty,z)}f\d\nu_j$ can be added. 
So, by \eqref{eq:g_y ineq}, $\nu_1(f)\ge\nu_2(f)$, for any $f\in\G$. 
In view of \eqref{eq:hG}, the proof of the ``if'' part of Theorem~\ref{th:dual cone} is now completed as well. 
\end{proof}


Theorem~\ref{th:dual cone} can be restated in the following ``basis'' form.   
\begin{theorem}\label{th:dual cone, b} 
Take any $(\nu_1,\nu_2)\in\NN_+(\G)\times\NN_+(\G)$, and also take any $s$ and $z$ in $I$. Then $(\nu_1,\nu_2)\in\hG$ if and only if all of the following conditions hold: 
\emph{ 
	\begin{enumerate}
	\item[(i')] $\nu_1(p_{s;0,i})=\nu_2(p_{s;0,i})\in\R$ 
	for all $i\in\intr0{k-1}$;  
	\item [(ii')]
	$\nu_1(p_{a,z;0:k:j})\ge\nu_2(p_{a,z;0:k:j})$ for all $j\in F_{k,n}$;  
	\item[(iii)] $\nu_1(p_{t;0,n}^+)\ge\nu_2(p_{t;0,n}^+)$ 
	for all $t\in I$. 
\end{enumerate}
}
\end{theorem} 
\noindent\rule{0pt}{1pt}\big(Recall here the definitions \eqref{eq:p_tjm} and \eqref{eq:p_{a,z;k:k:j}}-\eqref{eq:p_{a,z;i:k:j}}-\eqref{eq:F_kn}  concerning the $\w$-poly\-nomials $p_{s;0,i}$ for $i\in\intr0{k-1}$ and $p_{a,z;0:k:j}$ for $j\in F_{k,n}$.\big)

\begin{proof}[Proof of Theorem~\ref{th:dual cone, b}]
First here, by \eqref{eq:basis}, the $\w$-polynomials $p_{s;0,0},\dots,p_{s;0,k-1}$ constitute a basis of the linear space $\PP^{\le k-1}$, and so, condition (i') of Theorem~\ref{th:dual cone, b} is equivalent to condition (i) of Theorem~\ref{th:dual cone}.  

One can similarly see that the conjunction of conditions (i') and (ii') of Theorem~\ref{th:dual cone, b} is equivalent to condition (ii) of Theorem~\ref{th:dual cone} \big(which latter is in turn equivalent to the conjunction of conditions (i) and (ii) of Theorem~\ref{th:dual cone}\big). 

Alternatively, one can note that, by \eqref{eq:p_a in} and condition (ii) of Theorem~\ref{th:dual cone}, for each $(\nu_1,\nu_2)\in\NN_+(\G)\times\NN_+(\G)$ condition (ii') of Theorem~\ref{th:dual cone, b} is necessary for $\nu\in\hG$. 
On the other hand, condition (ii) of Theorem~\ref{th:dual cone} was used in the proof of the ``if'' part of Theorem~\ref{th:dual cone} only to obtain the conclusion \eqref{eq:P compar}. 
However, the same conclusion can be obviously obtained by using the definition of $P_{z,y}$ in 
\eqref{eq:P=,R=}, formula \eqref{eq:f^j(a+)}, and conditions (i') and (ii') of Theorem~\ref{th:dual cone, b} -- instead of \eqref{eq:P in} and condition (ii) of Theorem~\ref{th:dual cone}. 
Thus, Theorem~\ref{th:dual cone, b} is proved. 
\end{proof}

In the unit-gauges case, Theorem~\ref{th:dual cone, b} immediately results in the following corollaries, in view of \eqref{eq:p_t,unit}, \eqref{eq:p_a,unit,1}, \eqref{eq:p_a,unit,2}, and \eqref{eq:F_kn,unit}.

\begin{corollary}\label{cor:dual,powers,1} 
Suppose that $w_0=w_1=\dots=1$. 
Suppose also that $a=-\infty$. 
Take any $(\nu_1,\nu_2)\in\NN_+(\G)\times\NN_+(\G)$, and also take any real $s$ and $z$. Then $(\nu_1,\nu_2)\in\hG$ if and only if all of the following conditions hold: 
\emph{ 
	\begin{enumerate}
	\item[(i)] 
$\int_I(x-s)^i\,\nu_1(\d x)=\int_I(x-s)^i\,\nu_2(\d x)\in\R$ 
for all $i\in\intr0{k-1}$;  
	\item [(ii)]
	\rule{0pt}{10pt}$\int_I(x-z)^k\,\nu_1(\d x)\ge\int_I(x-z)^k\,\nu_2(\d x)$ if $k\le n$;   
	\item[(iii)] 
	\rule{0pt}{10pt}$\int_I(x-t)_+^n\,\nu_1(\d x)\ge\int_I(x-t)_+^n\,\nu_2(\d x)$ 
	for all $t\in I$. 
\end{enumerate}
}
\end{corollary}  

\begin{corollary}\label{cor:dual,powers,2} 
Suppose that $w_0=w_1=\dots=1$. 
Suppose also that $a>-\infty$. 
Take any $(\nu_1,\nu_2)\in\NN_+(\G)\times\NN_+(\G)$, and also take any real $s$. Then $(\nu_1,\nu_2)\in\hG$ if and only if all of the following conditions hold: 
\emph{ 
	\begin{enumerate}
	\item[(i)] 
$\int_I(x-s)^i\,\nu_1(\d x)=\int_I(x-s)^i\,\nu_2(\d x)\in\R$ 
for all $i\in\intr0{k-1}$;  
	\item [(ii)]
	\rule{0pt}{10pt}$\int_I(x-a)^j\,\nu_1(\d x)\ge\int_I(x-a)^j\,\nu_2(\d x)$ for all $j\in\intr kn$;   
	\item[(iii)] 
	\rule{0pt}{10pt}$\int_I(x-t)_+^n\,\nu_1(\d x)\ge\int_I(x-t)_+^n\,\nu_2(\d x)$ 
	for all $t\in I$. 
\end{enumerate}
}
\end{corollary}  

One may note here that the part of condition~(ii) in Corollary~\ref{cor:dual,powers,2} for $j=n$ follows from condition~(iii) there. Therefore, one may replace the specification $j\in\intr kn$ in condition~(ii) in Corollary~\ref{cor:dual,powers,2} by $j\in\intr k{n-1}$.

Corollary~\ref{cor:dual,powers,1} immediately results in the following statement, which will be useful in probabilistic applications such as ones considered in \cite{asymm,normal,left-arxiv}. 

\begin{corollary}\label{cor:dual,probab} 
Suppose that $w_0=w_1=\dots=1$. 
Suppose also that $a=-\infty$. 
Let $X$ and $Y$ be any r.v.'s with values in the interval $I$ and with distributions belonging to the admissible set 
$\NN_+(\G)$ (characterized in Proposition~\ref{prop:N_+^k}). Take any real $s$ and $z$. Then
\begin{equation*}
	\E f(X)\ge\E f(Y)\text{\quad for all }f\in\G
\end{equation*}
if and only if all of the following conditions hold: 
\emph{ 
	\begin{enumerate}
	\item[(i)] 
$\E(X-s)^i=\E(Y-s)^i\in\R$ 
for all $i\in\intr1{k-1}$;  
	\item [(ii)]
	$\E(X-z)^k\ge\E(Y-z)^k$  if $k\le n$;   
	\item[(iii)] 
	$\E(X-t)_+^n\ge\E(Y-t)_+^n$  
	for all $t\in I$. 
\end{enumerate}
}
\end{corollary} 

Clearly, similar probabilistic formulations 
can be immediately obtained based on Theorems~\ref{th:dual cone} and \ref{th:dual cone, b} and Corollary~\ref{cor:dual,powers,2}. 

By using the reflection transformation $\R\ni x\mapsto -x$, one immediately obtains the corresponding results for the ``reflected'' class of functions 
\begin{equation}\label{eq:F_-}
	\F_-^{k:n}(I):=\{f^-\colon f\in\F_+^{k:n}(-I)\}, 
\end{equation}
where $-I:=\{-x\colon x\in I\}$ and $f^-(x):=f(-x)$ for all $x\in I$. 
For instance, in view of part~(i) of Proposition~\ref{prop:N_+^k}, one has the following ``reflected'' counterpart of Corollary~\ref{cor:dual,probab}, where for simplicity we shall consider only the case when $k\le n$. 

\begin{corollary}\label{cor:dual,probab,refl} 
Suppose that $k\le n$ and $b=\infty$. 
Let $X$ and $Y$ be any r.v.'s with values in the interval $I$ and 
such that $\E p(X)\wedge\E p(Y)>-\infty$ for all (usual) polynomials $p$ of degree $\le k$ with $(-1)^k p^{(k)}\ge0$.  
Take any real $s$ and $z$. Then
\begin{equation*}
	\E f(X)\ge\E f(Y)\text{\quad for all }f\in\F_-^{k:n}(I)
\end{equation*}
if and only if all of the following conditions hold: 
\emph{ 
	\begin{enumerate}
	\item[(i)] 
$\E(X-s)^i=\E(Y-s)^i\in\R$ 
for all $i\in\intr1{k-1}$;  
	\item [(ii)]
	$\E(z-X)^k\ge\E(z-Y)^k$;   
	\item[(iii)] 
	$\E(t-X)_+^n\ge\E(t-Y)_+^n$  
	for all $t\in I$. 
\end{enumerate}
}
\end{corollary}


\subsection{Relations with Tchebycheff systems}\label{tcheb}

A result similar to a special case of Theorem~\ref{th:dual cone, b} was stated 
as two separate theorems in \cite[page~407]{karlin-studden}: Theorem~5.1 for $k=n+1$ and Theorem~5.2 for $k\in\intr1n$, in the notation of the present paper; the symbol $k$ in the present paper corresponds to $k+1$ in \cite[page~407, Theorem~5.2]{karlin-studden}. 
The latter two theorems in \cite{karlin-studden} are based on the papers \cite{karlin-novikoff} and \cite{ziegler}, respectively. 
In the case when $I$ is a finite open interval -- so that $I=(a,b)$, $-\infty<a<b<\infty$ -- 
Theorems~5.1 and 5.2 in \cite[page~407]{karlin-studden} characterize convex cones that are in a certain sense dual to the cones denoted in \cite{karlin-studden} as $C(u_0,\dots,u_n)$ and $\bigcap_{j=k}^n C(u_0,\dots,u_j)$, respectively, where $u_i$ coincides with the $\w$-polynomial $p_{a;0,i}$ and -- somewhat tacitly but crucially -- is assumed to be finite, for each $i\in\intr0n$; cf.\ the definition of the $u_i$'s in \cite[page~381, formula~(2.1)]{karlin-studden} and the definition of the $p_{t;j,m}$'s in \eqref{eq:p_tjm} (in the present paper). 
As stated in \cite[page~381]{karlin-studden}, for each $i$ the function $w_i$ is assumed to be of class $C^{n-i}$ on the closed interval $[a,b]$ (which is the closure of $I=(a,b)$ in $\R$). 
It is shown in \cite[page~379, Theorem~1.2]{karlin-studden} that these functions $u_0,\dots,u_n$ in \cite{karlin-studden} constitute a special case of an extended complete Tchebycheff system (\emph{ECT}-system) on $[a,b]$, as defined at the top of page~375 in \cite{karlin-studden}. 
By \cite[page~386, Theorem~2.1]{karlin-studden} and \eqref{eq:F}, the cone $C(u_0,\dots,u_n)$ in \cite{karlin-studden} essentially coincides (up to certain smoothness conditions) with our cone $\F_+^{n+1:n}$, and so, for $k\in\intr1n$ the cone $\bigcap_{j=k}^n C(u_0,\dots,u_j)$ essentially coincides with $\F_+^{k+1:n}$. 

There are a number of differences between Theorems~5.1 and 5.2 in \cite[page~407]{karlin-studden} and our Theorem~\ref{th:dual cone, b}. One is that the dual cone in \cite{karlin-studden} is defined in a traditional manner, as a set of signed measures $\nu$ rather than a set of of ordered pairs $(\nu_1,\nu_2)$ of nonnegative measures; cf.\ the beginning of the discussion in Subsection~\ref{admissible}; at that, the total variation of $\nu$ in Theorems~5.1 and 5.2 in \cite[page~407]{karlin-studden} was assumed to be finite. Also, as mentioned before, these theorems were obtained under the additional smoothness condition $w_i\in C^{n-i}[a,b]$ for all $i\in\intr0n$ (to be compared with the condition in the present paper that the $w_i$'s be Borel-measurable and locally bounded). 
Moreover, in these theorems in \cite{karlin-studden} the interval $I$ is assumed to be finite and open, whereas we allow $I$ to be any interval in $\R$. 
Further, our treatment appears to be more direct, as we define the classes $\F_+^{k:n}$ of multiply monotone functions directly in terms of the gauge functions $w_i$, rather than in terms of specific $\w$-polynomials $u_i=p_{a;0,i}$. 
Our method of proof is also more direct, without 
an explicit characterization or use of the extreme rays of the cones $\F_+^{k:n}$. 

However, the most significant difference between Theorems~5.1 and 5.2 in \cite[page~407]{karlin-studden} and our Theorem~\ref{th:dual cone, b} is that the former ones impose the mentioned additional, ostensibly innocuous condition of the finiteness of the functions $u_i=p_{a;0,i}$ for all $i\in\intr0n$. This additional condition rules out, among others, the unit-gauges case with $a=-\infty$ -- 
the most important case in applications such as ones considered in \cite{asymm,normal,left-arxiv}, which in fact motivated the present paper.

\section{Illustrations and applications}\label{appls}

\subsection{Solutions of compositional systems of linear differential inequalities}\label{diff-ineq}
The condition, in the definition \eqref{eq:F} of the class $\F_+^{k:n}$ of functions, that the generalized derivatives $f^{(j)}$ are nondecreasing for all $j\in\intr{k-1}n$ can be obviously transcribed as a system of linear differential inequalities $f^{(i)}\ge0$ for $i\in\intr k{n+1}$, where (say) $f^{(n+1)}(x):=\liminf_{y\downarrow x}\big(f^{(n)}(y)-f^{(n)}(x)\big)/(y-x)$ for $x\in I\setminus\{b\}$. 
In turn, the generalized derivatives $f^{(i)}$ can be obviously rewritten in terms of the usual derivatives $f,f',f'',\dots$ of the function $f$. 
Thus, ... provides a description of the convex cone of solutions of such systems of linear differential inequalities, and ... provides a dual description of the same. 


As an illustration of what has just been stated, turn here to  
a particular case of systems of gauge functions considered in the proof of Proposition~\ref{prop:M}. 
Namely, let $I=\R$, $k=2$, $n=5$, and $w_j(x)=e^{\la_j x}$ for $j\in\intr0\infty$, where $(\la_0,\dots,\la_5)=(0, 0, 0, -1, 2, 1)$. Here $x$ denotes any real number. In accordance with \eqref{eq:f^j}, for these gauge functions $w_j$ and function $f\in\L^n=\L^5$, one has the generalized derivatives 
\begin{multline*}
	(f_\w^{(0)}(x),\dots,f_\w^{(5)}(x))
	=\big(f(x),f'(x),f''(x),e^x f^{(3)}(x), \\ 
	e^{-x}(f^{(4)}(x)+f^{(3)}(x)),e^{-2 x}
   (f^{(5)}(x)-f^{(3)}(x))\big),  
\end{multline*}
where $f^{(j)}$ without the subscript ${}_\w$ denotes the usual $j$th derivative of $f$, with respect to the unit gauge functions. 
So, the condition $f\in\F_+^{k:n}$ means that $f\in\L^5$ is a solution to the following system of differential inequalities: 
\begin{multline}\label{eq:syst-ineqs}
	f'\ge0,\quad f''\ge0,\quad f^{(3)}\ge0,\quad  \\ 
	f^{(4)}+f^{(3)}\ge0,\quad f^{(5)}-f^{(3)}\ge0,\quad  f^{(6)}-2 f^{(5)}-f^{(4)}+2 f^{(3)}\ge0.  
\end{multline}
Here all the coefficients of the derivatives are constant, because the gauge functions $w_j$ are all exponential; for other gauge functions, that will not be the case. 

To solve this system of differential inequalities, note that, in view of \eqref{eq:p_tjm}, \eqref{eq:any M}, and  \eqref{eq:p_{a,z;k:k:j}}--\eqref{eq:p_{a,z;i:k:j}}, here 
\begin{equation}
	\big(p_{0;0,i}(x)\colon i\in\intr0{k-1}\,\big)=(1,x), 
\end{equation}
\begin{equation}
p_{t;0,n}(x)=\tfrac1{24}\,e^{2 t}\,  \big(4 ( 1 - e^{ t - x}) + (e^{2 (x - t)} - 1) - 
   12 (e^{x - t} - 1) + 6 (1 + x - t) (x - t)\big) 
\end{equation}
for $t\in\R$, 
\begin{equation}
	\big(p_{-\infty;k,j}(x)\colon j\in\intr kn\,\big)=(1,\infty,e^x/2,e^{2x}/6), 
\end{equation}
$F_{k,n}=F_{2,5}=\{k, k+2, k+3\}=\{2, 4, 5\}$, 
and 
\begin{equation}
	\big(p_{-\infty,0;0,k,j}(x)\colon j\in F_{k,n}\,\big)=\big(\tfrac{1}{2}\,x^2,\tfrac{1}{2}\,(e^x-1-x),\tfrac{1}{24}\, 
	(e^{2 x}-1-2x)\big).  
\end{equation}

It follows by ... that here any function $f$ that is a solution to the system \eqref{eq:syst-ineqs} of differential inequalities is the pointwise limit of a sequence of functions each of which is given by an expression of the form 
\begin{multline}
	c_0+c_1x+c_{2,4}e^x+c_{2,5}e^{2x} \\ 
	+c_{+,0,5}\big(4 ( 1 - e^{ t - x})_+ + (e^{2 (x - t)} - 1)_+ - 
   12 (e^{x - t} - 1)_+ + 6 (1 + x - t) (x - t)_+\big) \\ 
  +c_{-\infty,2,2}x^2+c_{-\infty,2,4}(e^x-1-x)+c_{-\infty,2,5}(e^{2 x}-1-2x) 
\end{multline}
for $x\in\R$, where $t\in\R$, $c_0$ and $c_1$ are in $\R$, and $c_{2,4},c_{2,5},c_{+,0,5},c_{-\infty,2,2},c_{-\infty,2,4},\break 
c_{-\infty,2,5}$ are in $[0,\infty)$.

%
 
Of course, using results on the dual cone in Section~\ref{dual G}, one can give a dual description of the cone of functions that are the solutions to system \eqref{eq:syst-ineqs} of differential inequalities.  

In view of \eqref{eq:f^j} and \eqref{eq:D^j,E^j}, the systems of differential inequalities discussed here are of the special, let us say compositional, form:  $E^i f\ge0$ for $i\in\intr k{n+1}$, where the linear differential operators $E^i$ are the cumulative compositions \break 
$E_i\cdots E_1E^0$ of linear differential operators of the first order. 
Compositional operators of the form $E^i=E^i_{w_0,\dots,w_i}$ are apparently not uncommon, especially for $i=2$. 
In particular, the operator $E^2_{w_0,w_1,w_2}$ with 
\begin{equation}\label{eq:w_0,w_1,w_2}
\text{$w_0=1$,\quad $w_1=\exp\int\frac{-2\mu}{\si^2}$,\quad $w_2=\frac2{\si^2w_1}$}	
\end{equation}
coincides with the linear operator $\Om$ given by the formula 
\begin{equation*}
	(\Om f)(x):=\mu(x)f'(x)+\frac{\si(x)^2}2\,f''(x), 
\end{equation*}
which is 
the infinitesimal generator of the semigroup corresponding to the diffusion described by the stochastic differential equation 
\begin{equation*}
\d X_t=\mu(X_t)\,\d t+\frac{\si(X_t)^2}2\,\d W_t, 	
\end{equation*}
where $W_\cdot$ is the standard Wiener process; cf.\ \cite[formula~(1.4)]{feller57}; the integral $\int\frac{-2\mu}{\si^2}$ in \eqref{eq:w_0,w_1,w_2} denotes an arbitrary anti-derivative of the function $\frac{-2\mu}{\si^2}$. 

The formal adjoint to $\Om$ is the Fokker--Planck operator, 
given by the formula 
\begin{equation*}
	(\Om^*f^*)(x):=-\big(\mu(x)f^*(x)\big)'+\frac{\big(\si(x)^2f^*(x)\big)''}2;  
\end{equation*}
the physical meaning of this expression is the rate of change in time at point $x$ of the concentration $f^*$ of the diffusing substance. See e.g.\ \cite{feller52,feller54,ventcel,feller57}.  
Observing that $(E^{2m}_{w_0,\dots,w_{2m}})^*=E^{2m}_{w_{2m},\dots,w_0}$ for nonnegative integers $m$ or by 
direct verification, we see that the Fokker--Planck operator is a compositional operator as well: 
\begin{equation*}
	\Om^*=E^2_{w_2,w_1,w_0}, 
\end{equation*}
with $w_0,w_1,w_2$ as in \eqref{eq:w_0,w_1,w_2}. 

In the special case with $\mu(x)\equiv-x$ and $\si(x)\equiv\sqrt2$, the diffusion operator $\Om$ is Stein's operator, which is the infinitesimal generator of the semigroup corresponding to the stochastic differential equation 
\begin{equation*}
\d X_t=-X_t\,\d t+\d W_t    	
\end{equation*} 
describing the standard Ornstein--Uhlenbeck process; 
see e.g.\ \cite{reinert05}. 
In this case, the gauge functions $w_0,w_1,w_2$ in \eqref{eq:w_0,w_1,w_2} can be chosen as follows: 
\begin{equation}\label{eq:stein w's}
	w_0=1,\quad w_1=1/\vpi,\quad w_2=\vpi, 
\end{equation}
where $\vpi$ is the standard normal probability density function.

\subsection{Refinements of the Chebyshev integral association inequality for generalized multiply monotone functions and related results}\label{cheb}
The classical Chebyshev integral association inequality states that 
\begin{equation}\label{eq:cheb}
	\int_0^1 f_1f_2\,\d\mu\ge\int_0^1 f_1\,\d\mu\int_0^1 f_2\,\d\mu,
\end{equation}
where $\mu$ is a Borel probability measure on $[0,1]$ and $f_1$ and $f_2$ are real-valued functions on $[0,1]$ that are both nondecreasing (or both nonincreasing); cf.\ e.g.\ \cite{graham83,tong97}. 

Andersson \cite{andersson58} obtained a counterpart of this inequality for convex functions, which may be stated as follows. If real-valued functions $f_1,\dots,f_N$ on $[0,1]$ are convex and non-negative and satisfy the condition $f_i(0)=0$ for all $i\in\intr1N$, then 
\begin{equation}\label{eq:anderss}
	\int_0^1f_1(x)\cdots f_N(x)\,\d x\ge\frac{2^N}{N+1}\,\int_0^1f_1(x)\,\d x\cdots\int_0^1f_N(x)\,\d x. 
\end{equation}
The equality in \eqref{eq:anderss} is attained when the functions $f_1,\dots,f_N$ are linear. 

Note that Andersson's conditions imply that the functions $f_1,\dots,f_N$ are also nondecreasing. Without the additional convexity condition, the best possible factor $\frac{2^N}{N+1}$ in \eqref{eq:anderss} would go down back to $1$.   

Andersson's result was extended by Karlin and Ziegler \cite{karlin-zieg75} to generalized multiply monotone functions in a variety of ways, but still on the interval $[0,1]$; results in \cite{karlin-zieg75} also refine ones by  Borell \cite{borell73} and Nehari \cite{nehari}. 

Using ..., one can obtain Chebyshev-type integral association inequalities for generalized multiply monotone functions on any interval $I\subseteq\R$. At that, as a further generalization of the mentioned results of \cite{andersson58,borell73,nehari,karlin-zieg75} and in keeping with \eqref{eq:cheb}, one may allow the integrals to be with respect to an arbitrary measure $\mu$ rather than with respect to the Lebesgue measure on $[0,1]$ (as e.g.\ in \eqref{eq:anderss}); this will be especially useful and natural when the interval $I$ is infinite. 


As an illustration, let $I=\R$, $k=n=1$, and $w_0(x)=\pi+\tan^{-1}x$ and $w_1(x)=1/(1+x^2)$ for $x\in\R$. For these gauge functions $w_0$ and $w_1$, the generalized derivatives $f^{(0)}$ and $f^{(1)}$ of a function $f\in\L^n=\L^1$ are as follows, in accordance with \eqref{eq:f^j}: 
\begin{equation}\label{eq:f^{(j)},cheb}
	f^{(0)}(x)=\frac{f(x)}{\pi+\tan^{-1}x}\quad\text{and}\quad 
	f^{(1)}(x)=\frac{f'(x)\left(1+x^2\right) \left(\pi+\tan^{-1}x\right) -f(x)}{\left(\pi+\tan^{-1}x \right)^2} 
\end{equation}
for real $x$, where $f'$ is the usual derivative of $f$. The class $\F_+^{k:n}=\F_+^{1:1}$ consists of all functions $f\in\L^1$ such that $f^{(0)}$ and $f^{(1)}$ are nondecreasing. 


\begin{theorem}\label{th:4/3}
Let the Borel probability measure $\mu$ on $\R$ be defined by the formula $\mu(\d x)=\frac{\d x}{\pi(1+x^2)}$, that is, $\mu(B)=\int_B\frac{\d x}{\pi(1+x^2)}$ for all Borel sets $B\subseteq\R$. 
Let $w_0$ and $w_1$ be the same gauge functions as in the preceding paragraph. 
Then, for any functions $f_1$ and $f_2$ in $\F_+^{1:1}$ such that $f_1(-\infty+)=f_2(-\infty+)=0$, 
\begin{equation}\label{eq:4/3}
	\int_\R f_1f_2\,\d\mu\ge\frac{384}{245}\,
	\int_\R f_1\,\d\mu\int_\R f_2\,\d\mu,
\end{equation}
and  
the constant factor $\frac{384}{245}=1.5673\dots$ is the best possible here. 
\end{theorem}

\begin{proof}[Proof of Theorem~\ref{th:4/3}]
Take indeed any functions $f_1$ and $f_2$ in $\F_+^{1:1}$ such that $f_1(-\infty+)=f_2(-\infty+)=0$. 
In particular, this implies that both $f_1$ and $f_2$ are nonnegative, since $f^{(0)}(x)=\frac{f(x)}{\pi+\tan^{-1}x}$ increases in $x\in\R$ for any $f\in\F_+^{1:1}$. 

Take any $\al\in\{1,2\}$ and any real $z$ and, in accordance with \eqref{eq:y}, let $y$ be an arbitrary real number in the interval $(-\infty,z]$. 
Let then $g_{\al,y}$ be the function in $\F_+^{1:1}$ constructed based on $f_\al$ the way the function $g_y$ in \eqref{eq:g_y} was constructed based on the function $f$. 
Note that any $\w$-polynomial of degree $\le0$ (which is of form $cw_0$ for a real constant $c$) is bounded in this context, because the function $w_0$ is bounded. Also, any function $f\in\F_+^{1:1}$ such that $f(-\infty+)=0$ is bounded on the interval $(-\infty,z)$, because the condition $f\in\F_+^{1:1}$ with $I=\R$ implies that $f$ is continuous on $\R$. Since 
$\mu$ is a probability measure, it follows by item \ref{bounds on g_y}(**) on page~\pageref{bounds on g_y} (with $k=1$) 
and dominated convergence that $\int_{(-\infty,z)} g_{\al,y}\,\d\mu\underset{y\downarrow-\infty}\longrightarrow\int_{(-\infty,z)} f_\al\,\d\mu\in[0,\infty)$. 
It also follows by 
item \ref{bounds on g_y}(*) on page~\pageref{bounds on g_y} and monotone convergence that 
$\int_{[z,\infty)} g_{\al,y}\,\d\mu\underset{y\downarrow-\infty}\longrightarrow\int_{[z,\infty)} f_\al\,\d\mu\in[0,\infty]$. So, 
\begin{equation}\label{eq:int conv}
\int_\R g_{\al,y}\,\d\mu\underset{y\downarrow-\infty}\longrightarrow\int_\R f_\al\,\d\mu\in[0,\infty]. 	
\end{equation}



By \eqref{eq:p_tjm}, \eqref{eq:F_kn}, and \eqref{eq:p_{a,z;k:k:j}}--\eqref{eq:p_{a,z;i:k:j}}, here 
\begin{equation}\label{eq:p's,cheb}
\left.
	\begin{aligned}
	p_{0;0,0}(x)&=w_0(x)=\pi+\tan^{-1}x, \\ 
	p_{t;0,1}(x)&=w_0(x)\int_t^x w_1(x_1)\,\d x_1=(\pi+\tan^{-1}x)(\tan^{-1}x-\tan^{-1}t), \\ 
	F_{1,1}&=\{1\}, \\ 
	p_{-\infty,0;0:1:1}(x)&=w_0(x)\int_0^x p_{-\infty;1,1}(x_1)\,\d x_1=(\pi+\tan^{-1}x)\tan^{-1}x
\end{aligned}
\right\}
\end{equation}
for all $t\in[-\infty,\infty)$ and $x\in\R$. So, by \eqref{eq:g_y=} and \eqref{eq:h_i,y}, for all real $x$ 
\begin{align*}
	g_{\al,y}(x)&=c_{0;\al,y}\, 
	w_0(x)+c_{1;\al}\,p_{-\infty,0;0:1:1}(x)
	+\int_{[y,\infty)}p_{t;i,n}^+(x)\,\nu_{\al}(\d t) \\ 
	&=c_{0;\al,y}\,(\pi+\tan^{-1}x)+c_{1;\al}\,(\pi+\tan^{-1}x)\tan^{-1}x \\
	&\ \ +(\pi+\tan^{-1}x)\int_{[y,\infty)}(\tan^{-1}x-\tan^{-1}t)_+\,\nu_{\al}(\d t), 
\end{align*}
where $c_{0;\al,y}\in\R$ depends only on $f_\al$ and $y$; $c_{1;\al}\in[0,\infty)$ depends only on $f_\al$; and the Borel measure $\nu_{\al}$ depends only on $f_\al$. 
The latter, displayed integral is finite for each $x\in\R$ \big(since $g_{\al,y}(x)$ is so\big) and for each $t\in[y,\infty)$ the integrand there monotonically decreases to $0$ as $x$ decreases to $-\infty$. So, by dominated convergence, this integral goes to $0$ as $x\to-\infty$. 
Taking also \eqref{eq:betw} (with $k=1$) into account, we have  
\begin{equation}
0=f_\al(-\infty+)\le g_{\al,y}(-\infty+)=c_{0;\al,y}\,\tfrac\pi2-c_{1;\al}\,\tfrac{\pi^2}4.  	
\end{equation} 

Recall definition \eqref{eq:f^j} of $f^{(j)}$ and, along with the functions $f_\al$ and $g_{\al,y}$, consider $f_\al^{(0)}=f_\al/w_0$ and $g_{\al,y}^{(0)}=g_{\al,y}/w_0$. 
Clearly, $f_\al^{(0)}(-\infty+)=\frac2\pi\,f_\al(-\infty+)=0$ and 
\begin{equation}\label{eq:c-c>0}
	g_{\al,y}^{(0)}(-\infty+)=\tfrac2\pi\,g_{\al,y}(-\infty+)=c_{0;\al,y}-c_{1;\al}\,\tfrac\pi2\ge0. 
\end{equation}
Take any real $\vp>0$. 
%
Since 
$f_\al^{(0)}(-\infty+)=0$, one has $f_\al^{(0)}(v_\vp)<\vp/2$ for some real $v_\vp$. 
By \eqref{eq:betw}, $g_{\al,y}^{(0)}(x)\sear{y\downarrow-\infty}f_\al^{(0)}(x)$ for all $x\in(-\infty,z]$. 
So, there is $y_\vp\in\R$ such that for all $y\in(-\infty,y_\vp]$ one has $g_{\al,y}^{(0)}(v_\vp)<f_\al^{(0)}(v_\vp)+\vp/2$. 
By \eqref{eq:g_y in}, $g_{\al,y}\in\F_+^{1:1}$, and so, the function $g_{\al,y}^{(0)}$ is nondecreasing. 
Hence, $g_{\al,y}^{(0)}(v)\le g_{\al,y}^{(0)}(v_\vp)<f_\al^{(0)}(v_\vp)+\vp/2<\vp$ for all $y\in(-\infty,y_\vp]$ and $v\in(-\infty,v_\vp]$. 
It follows that $g_{\al,y}^{(0)}(-\infty+)<\vp$ for all $y\in(-\infty,y_\vp]$. On the other hand, by \eqref{eq:c-c>0}, $g_{\al,y}^{(0)}(-\infty+)\ge0$. 
This shows that $g_{\al,y}^{(0)}(-\infty+)=c_{0;\al,y}-c_{1;\al}\,\frac\pi2\to0$ 
and hence $c_{0;\al,y}(\pi+\tan^{-1}x)\to c_{1;\al}\,\frac\pi2\,(\pi+\tan^{-1}x)$ uniformly in $x\in\R$ 
as $y\to-\infty$. So, letting 
\begin{align*}
	\hat g_{\al,y}(x)&:=g_{\al,y}(x)-c_{0;\al,y}(\pi+\tan^{-1}x) + c_{1;\al}\,\tfrac\pi2\,(\pi+\tan^{-1}x) \\ 
	&=c_{1;\al}(\pi+\tan^{-1}x)(\tfrac\pi2+\tan^{-1}x)
	+\int_{[y,\infty)}(\tan^{-1}x-\tan^{-1}t)_+\,\nu_{\al}(\d t)
\end{align*}
for all $x\in\R$, one has $\hat g_{\al,y}(-\infty+)=0$, $\hat g_{\al,y}\in\F_+^{1:1}$, and  
$\int_\R \hat g_{\al,y}\,\d\mu\underset{y\downarrow-\infty}\longrightarrow\int_\R f_\al\,\d\mu\in[0,\infty]$ by \eqref{eq:int conv}. 

Thus, it suffices to prove \eqref{eq:4/3} with $\hat g_{\al,y}$ in place of $f_\al$. 
Moreover, since each side of inequality \eqref{eq:4/3} is bilinear in $(f_1,f_2)$, 
it remains to check that 
\begin{equation}\label{eq:=4/3}
\min\big\{r(f_1,f_2)\colon f_1,f_2\text{ in the set }\{\rho\}\cup\{\tau_t\colon t\in\R\}\big\}=\frac{384}{245}, 	
\end{equation}
where 
\begin{equation}
	r(f_1,f_2):=\frac{\int_\R f_1f_2\,\d\mu}
	{\int_\R f_1\,\d\mu\,\int_\R f_2\,\d\mu}
\end{equation}
and 
\begin{equation}
	\rho(x):=(\pi+\tan^{-1}x)(\tfrac\pi2+\tan^{-1}x)\quad\text{and}\quad\tau_t(x):=(\pi+\tan^{-1}x)(\tan^{-1}x-\tan^{-1}t)_+
\end{equation}
for all real $x$ and $t$. 
Easy calculations show that $r(\rho,\rho)=\frac{384}{245}=1.5673\dots$; $r(\rho,\tau_t)$ is a rational function of $\tan^{-1}t$, which increases from $\frac{384}{245}$ to $\frac{18}7$ as $t$ increases from $-\infty$ to $\infty$; and 
$r(\tau_t,\tau_t)$ is a rational function of $\tan^{-1}t$, which increases from $\frac{384}{245}$ to $\infty$ as $t$ increases from $-\infty$ to $\infty$. 
This confirms \eqref{eq:=4/3} and thus completes the proof of Theorem~\ref{th:4/3}. 
\end{proof}

It should be clear now that one can  produce any number of results similar to Theorem~\ref{th:4/3}. 

Similar generalizations/refinements may be obtained for other inequalities between bilinear or, more generally, multilinear functionals 
on the Cartesian products of cones of generalized multiply monotone functions. One such potential extension concerns the Stolarsky inequality; see e.g.\ \cite{MR1931228} and further references therein.





\subsection{Generalized moment comparison inequalities for sums of independent random variables and (super)martingales}\label{compar}

As a quick illustration of how results on the dual cone in Section~\ref{dual G} can be used, consider the setting posed in Subsection~\ref{cheb}, with $I=\R$, $k=n=1$, and $w_0(x)=\pi+\tan^{-1}x$ and $w_1(x)=1/(1+x^2)$ for $x\in\R$, so that the class $\F_+^{k:n}=\F_+^{1:1}$ consists of all functions $f\in\L^1$ such that the functions $f^{(0)}$ and $f^{(1)}$ described in \eqref{eq:f^{(j)},cheb} are nondecreasing. Then, using Theorem~\ref{th:dual cone, b} and recalling also \eqref{eq:N_+^k}, \eqref{eq:G}, \eqref{eq:hG} \eqref{eq:p's,cheb}, one immediately obtains

\begin{proposition}\label{prop:comp}
For any real-valued r.v.'s $X$ and $Y$, the following conditions are equivalent to each other: 
\begin{enumerate}
	\item[(i)] Inequality  
\begin{equation}\label{eq:comp}
\E f(X)\le\E f(Y)	
\end{equation}
holds for all functions $f\in\L^1$ such that 
\begin{equation}
	\frac{f(x)}{\pi+\tan^{-1}x}\quad\text{and}\quad 
	\frac{f'(x)\left(1+x^2\right) \left(\pi+\tan^{-1}x\right) -f(x)}{\left(\pi+\tan^{-1}x \right)^2} 
\end{equation}
are nondecreasing in $x\in\R$. 
	\item[(ii)] Inequality \eqref{eq:comp} holds for all functions $f\in\{p_{0;0,0},p_{-\infty,0;0:1:1}\}\cup\{p_{t;0,1}^+\colon t\in\R\}$ given by formulas \eqref{eq:p's,cheb}. 
\end{enumerate}
\end{proposition}


Apparently more interesting and nontrivial existing and potential applications of duality results in Section~\ref{dual G} may be represented by the following theorem concerning 
normal domination of (super)martingales with conditionally bounded differences, which may be further applied to concentration of measure for separately Lipschitz functions, as shown in \cite[Section~4]{normal}. 
Let $(S_0,S_1,\dots)$ be a supermartingale relative to a filter $(H_{\le0},H_{\le1},\dots)$ of $\sigma$-algebras, with $S_0\le0$ almost surely (a.s.) and differences $X_i:=S_i-S_{i-1}$ for $i\in\intr1\infty$. 
Let $\E_j$ and $\Var_j$ denote the conditional expectation and variance, respectively, given $H_{\le j}$.

\begin{theorem}\label{th:bernoulli} 
Suppose that for every $i\in\intr1\infty$ there exist $H_{\le(i-1)}$-measur\-able r.v.'s $C_{i-1}$ and $D_{i-1}$ and a positive real number $s_i$ such that 
\begin{gather}
C_{i-1}\le X_i\le D_{i-1} \quad\text{ and } \label{eq:bern-cond1} \\
D_{i-1}-C_{i-1}\le 2 s_i \label{eq:bern-cond2}
\end{gather}
a.s. Then for all $f\in\F_+^{1:5}$ (in the unit-gauges case) and all $n\in\intr1\infty$
\begin{equation}\label{eq:bernoulli}
\E f(S_n)\le\E f(sZ),
\end{equation}
where 
$$s:=\sqrt{s_1^2+\dots+s_n^2}$$
and $Z\sim N(0,1)$. 

If, moreover, $(S_0,S_1,\dots)$ is a martingale relative to $(H_{\le0},H_{\le1},\dots)$ with $S_0=0$ a.s., then inequality \eqref{eq:bernoulli} holds for all $f\in\F_+^{2:5}$. 
\end{theorem}

\begin{proof}[Proof of Theorem~\ref{th:bernoulli}]
By \cite[Theorem~2.1]{normal}, $\E(S_n-t)_+^5\le\E(sZ-t)_+^5$ for all $t\in\R$. 
Also, the conditions that $S_0\le0$ a.s. and $(S_0,S_1,\dots)$ be a supermartingale, yield $\E S_n\le0=\E sZ$. So, by Corollary~\ref{cor:dual,probab} with $I=\R$, 
inequality \eqref{eq:bernoulli} holds for all $f\in\F_+^{1:5}$. 

Assuming now that $(S_0,S_1,\dots)$ is a martingale with $S_0=0$, one has $\E S_n=0=\E sZ$. Also, it then follows that $\E S_n^2=\sum_1^n\E X_i^2=\sum_1^n\E\Var_{i-1} X_i
\le\sum_1^n\E|C_{i-1}D_{i-1}|\le\frac14\,\sum_1^n\E(|C_{i-1}|+|D_{i-1}|)^2
=\frac14\,\sum_1^n\E(D_{i-1}-C_{i-1})^2\le\sum_1^n s_i^2=s^2=\E(sZ)^2$;  
the first inequality in the above chain of equalities and inequalities follows by \cite[(2.12)]{normal} and the third equality in this chain follows because, by \eqref{eq:bern-cond1}, $C_{i-1}\le\E_{i-1}X_i=0\le D_{i-1}$ a.s. 
So, under the additional conditions stated in the last sentence of Theorem~\ref{th:bernoulli}, \eqref{eq:bernoulli} holds indeed for all $f\in\F_+^{2:5}$. 
\end{proof}

Recalling 
the definition of the set $\H_+^{i:n}$ in the beginning of Subsection~\ref{H_+} and Proposition~\ref{prop:H in F}, one sees that   
Theorem~\ref{th:bernoulli} provides an extension of inequality \eqref{eq:bernoulli}, from all $f\in\H_+^{0:5}$ to all $f\in\F_+^{1:5}$ \big(or, under the additional, martingale conditions on $(S_0,S_1,\dots)$, even to all $f\in\F_+^{2:5}$\big). 
Quite similarly one can extend other results in \cite{normal}, including (i) Theorem~2.6 as far as it concerns (2.3); (ii) inequality (4.3); (iii) Theorem~4.4 and Corollary~4.8 as far as they concern (4.3) -- all the references in this sentence are to \cite{normal}. 

\begin{center}
	***
\end{center}


Among other applications is the main result in \cite{left-arxiv}, whose proof is based in part on Corollary~\ref{cor:dual,probab,refl} in the present paper. In particular, that result in \cite{left-arxiv} implies the following. 

\begin{theorem}\label{th:left}\ 
Let $X_1,\dots,X_n$ be any \emph{nonnegative} independent r.v.'s such that for some nonnegative real numbers $m,m_1,\dots,m_n,s,s_1,\dots,s_n$ one has $0<s\le m^2/n$, 
\begin{equation*}\label{eq:>m_i,<s_i}
	\E X_i\ge m_i \quad\text{and}\quad\E X_i^2\le s_i  
\end{equation*}
for all $i\in\intr1n$ and 
\begin{equation*}\label{eq:>m,<s}
	m_1+\dots+m_n\ge m\quad\text{and}\quad s_1+\dots+s_n\le s. 
\end{equation*}
Let $Y_1,\dots,Y_n$ denote any independent identically distributed r.v.'s such that 
\begin{equation*}
	\P(Y_1=\tfrac sm)=1-\P(Y_1=0)=\tfrac{m^2}{ns}.  
\end{equation*}
Then 
\begin{align*}
	\E f(S_n)&\le\E f\big(Y_1+\dots+Y_n\big)  
	\le\E f\big(\tfrac sm\Pi_{m^2/s}\big) 
	\le\E f\big(m+Z\sqrt s\,\big) 
\end{align*}
for all $f\in\F_-^{1:3}$, where $\Pi_\la$ is any r.v.\ having the Poisson distribution with parameter $\la\in(0,\infty)$ and $Z$ is any standard normal r.v.  
\end{theorem}

Variants of this result for the classes $\F_-^{2:2}$, $\F_-^{3:2}$, $\F_-^{2:3}$, $\F_-^{3:3}$, and $\F_-^{4:3}$ of functions in place of $\F_-^{1:3}$ are given in \cite[Remark 2.5]{left-arxiv}.   

Similarly, it would be quite convenient to use Corollary~\ref{cor:dual,powers,1} of the present paper in place of Lemmas~1--4 in \cite{asymm} and in place of Lemmas 4.5 and 4.6 in \cite{pin-hoeff-arxiv-reftoAIHP}; those lemmas in \cite{asymm,pin-hoeff-arxiv-reftoAIHP} were proved in a rather ad hoc manner. 


\bibliographystyle{abbrv}


\bibliography{C:/Users/ipinelis/Dropbox/mtu/bib_files/citations12.13.12}


\end{document}